\documentclass[a4paper]{article}
\usepackage{amsthm,amsfonts,amsmath,amssymb}
\usepackage[cp1251]{inputenc}
\usepackage[english]{babel}
\usepackage[final]{graphicx}
\usepackage{setspace}
\usepackage{xcolor,colortbl}
\usepackage{hhline}
\usepackage{tikz}
\begin{document}
 \newtheorem{theorem}{Theorem}[section]
 \newtheorem{lemma}[theorem]{Lemma}
 \newtheorem{proposition}[theorem]{Proposition}
 \newtheorem{property}[theorem]{Property}
 \newtheorem{corollary}[theorem]{Corollary}
 \theoremstyle{remark}
 \newtheorem{conjecture}{\bf Conjecture}
 \newtheorem{remark}{\bf Remark}
 \newtheorem{Example}{\bf Example}
 \newtheorem{question}{\bf Question}

\title{Perfect $2$-colorings of Hamming graphs}
\author{Evgeny A. Bespalov$^1$, Denis S. Krotov$^1$, Aleksandr A. Matiushev$^2$, \\ Anna A. Taranenko$^1$, Konstantin V. Vorob'ev$^1$}
\date{$^1$ Sobolev Institute of Mathematics, Novosibirsk, Russia \\ $^2$ Novosibirsk State University, Novosibirsk, Russia}

\maketitle

\begin{abstract}
We consider the problem of existence of perfect $2$-colorings (equitable $2$-partitions) of Hamming graphs with given parameters. We start with conditions on parameters of graphs and colorings that are necessary for their existence.  Next we observe known constructions of perfect colorings and propose some new ones giving new parameters. At last, we deduce which parameters of colorings are covered by these constructions and give tables of admissible parameters of $2$-colorings in Hamming graphs $H(n,q)$ for small $n$ and $q$.
Using the connection with perfect colorings, we construct an orthogonal array OA(2048,7,4,5).

\end{abstract}


\section{Introduction}
An \emph{equitable $k$-partition}
of a graph $G=(V,E)$ (in general, a multigraph, i.e., loops and multiedges are allowed)
is a partition of the vertex set $V$ into 
$k$ nonempty \emph{cells} 
$V_1$, \ldots, $V_{k}$
such that every cell induces a regular subgraph
and the bipartite graph formed by the edges between
any two different cells is biregular.
Another equivalent name of such a partition is a \emph{perfect $k$-coloring},
which, formally, refers to a $k$-valued function on 
$V$ that defines an equitable partition: 
each cell of the partition is the preimage of some value.

In the current paper, we study perfect $2$-colorings
of the $q$-ary Hamming graph $H(n,q)$, mainly
focusing on the nonbinary case $q>2$ and on constructions that give new parameters of perfect $2$-colorings. 
A characterization of completely regular codes (including perfect $2$-colorings)
in $H(n,q)$ with the second largest eigenvalue $\theta_1 = n(q-1) - q$ is obtained in~\cite{Meyer:cyccond} (where the same question is also solved for Johnson and regular complete multipartite graphs), and perfect $2$-colorings of $H(n,q)$ with the third largest eigenvalue $\theta_2 = n(q-1) - 2q$ are described in~\cite{MV:eq2eigen}.
Parameters of $2$-colorings of the binary
Hamming graph $H(n,2)$, the hypercube,
were studied in 
\cite{Canogar:2000,FDF:CorrImmBound,FDF:PerfCol,FDF:12cube.en,KroVor:2020}. 
In \cite{FDF:PerfCol}, general constructions of
perfect $2$-colorings of the hypercube were described
(we also mention the conference 
work~\cite{Canogar:2000},
where some of the existence results were announced without proof). In \cite{FDF:CorrImmBound}, a bound on the correlation
immunity of boolean functions was proved, giving
a powerful non-feasibility test for parameters of perfect $2$-colorings of the hypercube. 
In \cite{FDF:12cube.en}, two special parameter sets
of perfect $2$-colorings of $H(n,2)$ were considered;
in one case, perfect colorings were constructed,
in the other case, the nonexistence was proved. 
In \cite{KroVor:2020}, 
the infeasibility of an infinite series of parameters of $2$-colorings in $H(n,2)$ was proved.
In this paper, we consider generalizations of 
the constructions from \cite{FDF:PerfCol} to the non-binary case $q>2$
and some new constructons for non-prime $q$.
The well-known connection of perfect colorings and orthogonal arrays
plays an important role in our study and, in particular,
allows us to construct new orthogonal arrays.
The perfect $2$-colorings 
 are equivalent 
 to a special case of 
 completely regular codes,
 namely, to the completely regular codes 
 of covering radius $1$.
 In a recent survey \cite{BRZ:CR}
 on completely regular codes,
 Borges, Rif\`a, and Zinoviev
 mentioned perfect $2$-colorings saying that 
 very little is known for $q>2$.
 Since there are more questions than answers in this area,
 our current work can be considered as a step forward
 from ``very little'' to ``little''.

The Hamming graphs belong to the class of distance-regular graphs, 
and we note that perfect colorings of distance-regular graphs attract 
much attention by several reasons.
Many classes of combinatorial configurations can be defined 
as perfect colorings with special parameters.
Examples of such configurations are perfect codes,
$t$-$(v,t+1,\lambda)$- and $t$-$(v,t+2,1)$-designs \cite[p.178]{Martin:94}, 
latin squares and latin hypercubes (equivalently, distance-$2$ MDS codes), 
MDS codes with distance $3$ \cite[Corol.\,6]{KKO:smallMDS}, transversals in latin squares.
In some cases,
optimal, by mean of some bound, objects are proven
to be in a one-to-one correspondence with 
perfect colorings with special parameters.
For example, unbalanced boolean functions attaining
the correlation-immunity bound \cite{FDF:CorrImmBound};
almost perfect codes and
some classes of optimal codes (see e.g. \cite{Kro:2m-3}, \cite{Kro:2m-4});
orthogonal arrays attaining the Bierbrauer--Friedman bound \cite{Bierbrauer:95,Friedman:92} induce perfect $2$-colorings
of  the Hamming graph
\cite{Potapov:2010}, \cite{Pot:2012:color};
binary orthogonal arrays attaining the
Bierbrauer--Gopalakrishnan--Stinson bound \cite{BGS:96}
induce perfect $3$-colorings
of the hypercube \cite{Kro:OA13}.
A very nice result of Potapov~\cite{Pot:2012:color}
shows a one-to-one correspondence between
the perfect $2$-colorings of the Hamming graph $H(n,q)$ and the boolean-valued functions on $H(n,q)$
attaining a bound that connects the correlation immunity of the function, the density of ones, 
and the average $0$-$1$-contact number
(the number of neighbors with function value $1$ for a given vertex with value $0$).
Besides the Hamming graphs, distance-regular graphs where perfect colorings
have been studied include Johnson graphs, see e.g. \cite{AvgMog:J63J73,GavGor:2013}, 
Latin-square graphs \cite{BCGG:2018}, 
halved hypercubes \cite{KMV:halved},
Grassmann graphs $G_q(n,2)$, see e.g. \cite{dWM:2020} and \cite{Matkin:2018}.
In finite geometry, perfect $2$-colorings are studied as intriguing sets, see e.g.~\cite{BLP:09}.
As an example of the study of perfect colorings of non-distance-regular graphs
yielding interesting and deep results we refer to \cite{Puz:2011en}.

Let us describe the structure of the paper. In Section~\ref{s:pre}, we introduce the main concepts related to perfect colorings in Hamming graphs and provide some easy observations. 

In Section~\ref{nessec}, we study conditions necessary for the existence of perfect colorings with given parameters.
Since Hamming graphs are distance-regular, 
we can specialize for them general results on
weight distributions of colors 
in distance-regular graphs. 
We also consider some other necessary
conditions based on algebraic 
and arithmetic properties 
of colorings and distributions 
of colors in faces of Hamming graphs.

Section~\ref{consrsec} is devoted to constructions of perfect colorings. Firstly, we give a series of constructions based on coverings of Hamming graphs. 
Next, we consider three direct constructions of perfect colorings based on
additive codes 
(Section~\ref{s:additive}, where, as a special case, we construct 
an orthogonal array OA$(2048,7,4,5)$), 
MDS codes, and $1$-perfect codes (Section~\ref{perfcodes}). 
Then, in Section~\ref{flaasssec}, we describe more complex constructions of perfect colorings combining some of the previous ones and providing new admissible parameters of colorings.

We summarize results on the admissibility of parameters of perfect colorings in Section~\ref{paramsec}. 
In particular, for cases when $q$ 
is a prime power we find sufficient conditions for the existence of $(b,c)$-coloring in $H(n,q)$ for $n$  greater or equal some $n_0$ and estimate the minimum value $n_0$ for which such colorings exist.

 Finally, in Appendix we provide tables 
 with admissibility statuses for parameters
 of colorings in $q$-ary $n$-dimensional
 Hamming graphs for small $n$ 
 and $q = 3$, $4$, and $6$.  
 
  We emphasize that all the results of our paper 
 that are applicable to the case $q=2$ 
 were proved for this case 
 by Fon-Der-Flaass \cite{FDF:PerfCol}.
 Moreover, some binary results 
 (bounds in \cite{FDF:CorrImmBound} and 
 \cite{FDF:PerfCol}, possibility to partition one color into lines in the main construction of \cite{FDF:PerfCol}, 
 special cases in \cite{FDF:12cube.en})
 were not generalized or generalized partially.
 On the other hand, new results for composite $q$ have no analogs 
 in the binary case.


\section{Notations, definitions, and easy observations}\label{s:pre}

Given $n$ and $q$, the \emph{Hamming graph $H(n,q)$} is the graph with the vertex set $\mathbb{Z}_q^n = \{ x= (x_1, \ldots, x_n)\mid x_i \in \mathbb{Z}_q \}$ such that vertices $x$ and $y$ are adjacent if and only if the Hamming distance (the number of different positions) $d(x,y)$ between them is $1$.
Observe that $H(n,q)$ 
is a connected $n(q-1)$-regular graph.
For any vertex $x$, its (Hamming) 
\emph{weight} $\mathrm{wt}(x)$ is defined to be the Hamming distance
from $x$ to the all-zero vertex $0^n$.
The binary 
Hamming graph $H(n,2)$ is also known as 
a \emph{hypercube}, or the \emph{$n$-cube}.

The \emph{spectrum} of the Hamming graph  $H(n,q)$
is the eigenspectrum of its adjacency matrix; 
it consists of the eigenvalues 
$\theta_i = n(q-1) - qi$ 
with multiplicity $\binom{n}{i} (q-1)^i$,
$i = 0, \ldots, n$, see e.g. \cite{Brouwer}.

For $k \in \{0, \ldots, n\}$, a
\emph{$k$-dimensional face}, or simply 
a \emph{$k$-face}, is a set of vertices
of $H(n,q)$ that induces a subgraph
isomorphic to $H(k,q)$.  
Faces of dimension $1$ are
essentially the maximal cliques 
in $H(n,q)$
and
often called \emph{lines}.

A multiset $C$ of vertices of $H(n,q)$
is called an \emph{orthogonal array}
of strength $t$, or OA$(|C|,n,q,t)$
if every $(n-t)$-face contains 
$|C|/q^t$ elements of $C$. 
In connection with perfect colorings, 
we will consider only \emph{simple} orthogonal arrays,
without multiple elements.

A \emph{perfect $k$-coloring} 
of a graph $G$ 
is a surjective function $f$
from the vertex set
to the set of \emph{colors} $\{ 1, \ldots, k\}$ 
such that  each vertex $x$ 
of color $i$ is adjacent
to exactly $s_{i,j}$ vertices
of color $j$, where $s_{i,j}$ is some constant 
 that does not depend on the choice of $x$. 
 The numbers $s_{i,j}$ are the \emph{parameters} of the perfect coloring and they  are the degrees of the regular and biregular subgraphs induced by the color sets.
The matrix $S = ( s_{i,j})$ 
of order $k$ 
is called the \emph{quotient matrix} 
of the perfect coloring $f$.

Let us have a closer look at parameters of perfect $2$-colorings of $H(n,q)$.
The quotient matrix of a perfect $2$-coloring $f$ is usually written as
$$S = \left( \begin{array}{cc}
a & b \\ c & d
\end{array} \right).$$
It  has two different eigenvalues: the \emph{trivial} eigenvalue $n(q-1)$ (the graph degree) and the \emph{second} eigenvalue $\theta = n(q-1) - (b + c) = a-c$. 
Since $a + b = c + d = n(q-1)$, 
 the parameters $b$ and $c$
uniquely define the matrix. 
Thus, we will say that $f$ 
is a \emph{$(b,c)$-coloring} 
if it has the quotient matrix $S$ as above.
Rearranging colors, 
we may assume that $b \geq c$. 

It is easy to prove 
that the number of vertices 
of the first color
in a perfect $(b,c)$-coloring 
of $H(n,q)$ is  equal to
$\frac{c}{b+c} \cdot  q^n$, 
while for  the second color it is 
$ \frac{b}{b+c} \cdot q^n$. 
Reducing factors, 
we find that the densities 
of the colors in the perfect
$(b,c)$-coloring $f$ 
are $\frac{c}{b+c}$  
and $\frac{b}{b+c}$ respectively,
or, equivalently, 
$\frac{c'}{b'+c'}$  
and $\frac{b'}{b'+c'}$,
where $b' =\frac{b}{\gcd(b,c)}$
and $c' =\frac{c}{\gcd(b,c)}$. 
Later we will see that the
parameters $b'$ and $c'$ 
not only define
the proportion of colors 
in a perfect $(b,c)$-coloring 
but play an important role
in the characterization
of admissible parameters of colorings.

At the end of this section,
we recall one 
straightforward application
of K\"onig's theorem,
which will be used 
in one of the constructions.

\begin{proposition}[\cite{FDF:PerfCol}] \label{dcmpboolclr}
Let $f$ be a perfect coloring of $H(n,2)$,
and let $i$ be one of the colors.
If $f^{-1}(i)$ is not an independent set, 
then it can be partitioned into
edges (lines). 
\end{proposition}


\section{Necessary conditions on parameters of perfect colorings} \label{nessec}

We start with one simple algebraic condition on parameters of $(b,c)$-colorings.
It is well known \cite[Ch.\,5, Lemma 2.2]{Godsil93} that the spectrum of the quotient matrix of a perfect coloring in a graph is contained 
in the graph spectrum.
Since all eigenvalues of the Hamming graph $H(n,q)$ 
have the form $\theta_i = n(q-1) - qi$,
where $i = 0, \ldots, n$,
and the second eigenvalue of a $(b,c)$-coloring of $H(n,q)$ 
is $\theta = n(q-1) - (b+c)$, we have the following.

\begin{proposition}
If there exists a $(b,c)$-coloring 
of $H(n,q)$, then 
$b + c = q i$ 
for some $i \in \{ 1, \ldots, n\}$.
\end{proposition}

\subsection{Perfect colorings and orthogonal arrays}

 For perfect colorings of $H(n,q)$, 
 it is known that all large faces 
(of dimension greater than 
$(n+\theta)/q$ where $\theta$ is the second largest 
eigenvalue of the quotient matrix)
 have the same densities of colors as in the whole graph, 
 {see e.g. \cite[Proposition~1]{Kro:OA13}, 
 which is essentially a special case 
 of~\cite[Theorem~4.4]{Delsarte:1973}}. 
 In particular, 
 for $2$-colorings we have

\begin{proposition} \label{corrimmune}
Each color of a $(b,c)$-coloring 
of $H(n,q)$ forms an orthogonal array of strength
$\frac{b+c}{q} -1 $.
\end{proposition}

Based on this fact,
we deduce the following 
necessary condition 
on the quotient matrices of 
$2$-colorings. 
A special case of this condition 
was considered 
in \cite{HedRoos:2011} 
to show the nonexistence
of some $1$-perfect codes
over non-prime-power alphabets.

\begin{theorem} \label{corbound}
Assume that there exists a perfect $(b,c)$-coloring
of $H(n,q)$.
Then $b' + c'$ divides 
the cardinality $q^k$ of a $k$-face,
where $k={n - \frac{b+c}{q} +1}$, 
$b'=\frac{b}{\gcd(b,c)}$, 
and $c'=\frac{c}{\gcd(b,c)}$.
\end{theorem}

\begin{proof}
 By Proposition~\ref{corrimmune}, 
 for a $(b,c)$-coloring
 the densities of colors in each $k$-face
 are the same as in $H(n,q)$.   
 Recall that the colors  
 have densities  $\frac{b'}{b'+c'}$  
 and $\frac{c'}{b'+c'}$. 
 Since each $k$-face has 
 an integer number of vertices 
 of each color, 
 we conclude that $b' + c'$ 
 divides 
 $q^k$.
\end{proof}

If $c=1$, 
we can strengthen the conclusion of the theorem.

\begin{theorem} \label{c1cond}
Assume that there exists a perfect $(b,1)$-coloring
of $H(n,q)$. 
Then $b$ is divisible by $q-1$. 
Moreover, if $q = p^s$ 
for some prime $p$,
then $b + 1 = q^r$ 
for  some $r \in \mathbb{N}$.
\end{theorem}

\begin{proof}
For a $(b,1)$-coloring $f$,
each vertex of the second color has exactly one neighbor 
of the first color. 
It yields  that 
each line 
of $H(n,q)$ contains 
$0$, $1$ or $q$ vertices 
of the first color and
$q$, $q-1$ or $0$ vertices 
of the second color.
In particular, 
every line containing 
a given vertex 
of the first color 
has either $q-1$ or $0$ vertices 
of the second color.
So $b$ is divisible by $q-1$.

Let $b = m(q-1)$ for some 
$m \in \mathbb{N}$. 
Suppose that $q = p^s$ 
for some prime $p$.  
Since $b$ and $c=1$ are relatively prime,
Theorem~\ref{corbound} implies that 
$b + 1$ divides $q^r$ 
for some $r \in \mathbb{N}$ 
and $b + 1 = p^j$ 
for some $j \in \mathbb{N}$. 
It is not hard to see that  
$b = p^j - 1$  
is divisible by 
$q-1 = p^s - 1$ if and only 
if $j = rs$ for some $r \in \mathbb{N}$.
Thus  $b + 1 = q^r$.
\end{proof}

In hypercubes $H(n,2)$, 
there is an additional 
necessary condition 
on parameters of 
$(b,c)$-colorings.

\begin{lemma}[\cite{FDF:CorrImmBound}] \label{Fcorbound}
If there exists a $(b,c)$-coloring 
of $H(n,2)$ 
with $b \neq c$, 
then $n \geq \frac{3}{4}(b+c)$.
\end{lemma}

This bound is a special case of the 
bound on the correlation immunity 
of boolean functions (equivalently, the strength
of simple binary orthogonal arrays) proved
in~\cite{FDF:CorrImmBound} and extended to
non-simple orthogonal arrays in~\cite{Khalyavin:2010.en}.
Generalizing this bound to $q>2$
is an open research problem.

\begin{conjecture}
If there exists a $(b,c)$-coloring 
in the Hamming graph $H(n,q)$ 
with $b' + c' > q$,
where $b'=\frac{b}{\gcd(b,c)}$,
$c'=\frac{c}{\gcd(b,c)}$,
then  $n \geq \frac{q+1}{q^2}(b+c)$.
\end{conjecture}

\begin{conjecture} 
Let  
$C$ be an OA$(|C|,n,q,t)$.
If (a) $|C|<q^{n-1}$ or (b) $C$ is simple and 
$|C|$ is not divisible by $q^{n-1}$, 
then $n-t \geq \frac{n}{q+1}+1$.
\end{conjecture}

\subsection{Weight distribution theorems and generalizations}

In this section, we briefly discuss very strong properties of perfect colorings known as the distance invariance 
and its generalizations. In our experience, 
we could not find examples
when these invariants
reject some putative quotient
matrix of a perfect $2$-coloring 
of Hamming graphs in a non-binary case.
However, they do work for that purpose when the number of colors is larger than $2$, 
and the potential of the development of the theory makes these methods necessary 
to mention here. 
The general method to reject a putative quotient matrix
using distance invariants is the following: 
if there is a way to calculate the number 
of groups vertices with some properties 
using only the parameters of the graph 
and the quotient matrix, but the calculations 
result in a negative or non-integer number,
then perfect colorings with such a quotient matrix do not exist.

The simplest distance-invariance property says that the weight distributions of all colors
of a perfect coloring depend only on 
the color of the initial vertex 
and
the parameters of the graph and the coloring.
(The weight distribution is the multiset of distances from the initial vertex to the vertices of the given color.)
This property holds for perfect colorings of any distance-regular graph and
generalizes
the well known Shapiro--Slotnick--Lloyd
theorem~\cite{Lloyd,ShSl} on the distance invariance of perfect codes.
In its turn, it can be generalized to the possibility to calculate the weight distribution
of the coloring with respect to any completely regular code, not only a single vertex,
see e.g.~\cite{Kro:struct}.

Another generalization of the distance invariance \cite{Vas09:inter,Vas:2019:local,Hyun:2012:local}
connects the weight distributions of a perfect coloring
of a Hamming graph in two orthogonal faces (so-called local weight distributions).
Inexplicitly, this connection was utilized in the bounds 
\cite[Theorems~1,2]{FDF:PerfCol} for perfect $2$-colorings of $H(n,2)$.

The strongest of known generalizations~\cite{Vas09:inter} of the distance invariance
 enables one to calculate (see formulas in~\cite{Kro:interw}) the multi-parameter invariants called 
interweight distributions of perfect colorings. 
These invariants work well for proving the nonexistence
of perfect colorings in $H(n,q)$ if $q=2$, see e.g.~\cite{CRCtable},
but examples~\cite{Kro:interw} show that their natural generalization
is not invariant in the non-binary case.
The existence of similar strong invariants for $q>2$ is an open problem.


\section{Constructions of perfect colorings} \label{consrsec}

\subsection{Covering-based constructions} \label{coverconstrsec}

In this section,
we consider a series of constructions 
that allow us to obtain perfect colorings 
of a Hamming graph 
on the base of perfect colorings 
of a smaller Hamming graph. 
All these constructions
utilize graph coverings. 
In order to describe them,
it is convenient  
to represent Hamming graphs 
as Cayley graphs.

Let $\Gamma$ be a finite group and 
$A \subset \Gamma$ be a (multi)set 
of $\Gamma$ such that $A = A^{-1}$. 
The \emph{Cayley (multi)graph
$\mathrm{Cay}(\Gamma,A)$ with the connecting set $A$}
is a (multi)graph with the vertex set 
$\Gamma$ and the edge set 
$\{ (g,ga) \mid  g \in \Gamma, a \in A \}.$
For convenience, we assume everywhere that
$\Gamma$ is an abelian group.

The Hamming graph $H(n,q)$ 
is  the Cayley graph $\mathrm{Cay}(\mathbb{Z}_q^n, \mathcal{I}(n,q))$ 
with the connecting set $\mathcal{I}(n,q)$ 
consisting of all vertices of weight $1$.

A multigraph $G = (V,E)$ 
is said to \emph{cover} 
a multigraph $H = (U,W)$ 
if there exists 
a surjective function 
$\varphi : V \rightarrow U$,
called a \emph{covering},
such that for each $x \in V$ 
the equality 
$\{\varphi(y) \mid  (y,x) \in E  \} 
 = \{ w  \mid  (w,\varphi(x)) \in W\}$ 
holds as for multisets.
It is straightforward 
that a covering is 
a perfect coloring of $G$ 
with the quotient matrix $S$
equal to the adjacency matrix of $H$.  
Here under the adjacency matrix $M$ of a multigraph $H$ 
we mean a matrix with entries $m_{i,j}$ equal to 
the number of edges between vertices $i$ and $j$ 
and entries $m_{i,i}$ equal to the number of loops.

There are natural 
coverings of Cayley graphs 
based on the following straightforward and well-known fact. 
Recall that $\varphi: \Gamma \rightarrow \Gamma'$ 
is a \emph{homomorphism} 
between groups  $(\Gamma,*)$ 
and $(\Gamma',\star)$
if for all $g,h \in \Gamma$ 
we have 
$\varphi(g * h) 
= \varphi(g) \star \varphi(h)$.

\begin{lemma} \label{homocover}
Let 
$\varphi: \Gamma \rightarrow \Gamma'$ 
be a surjective homomorphism between groups 
$\Gamma$ and $\Gamma'$, 
and let $A \subset \Gamma$ 
be a multiset such that $A = A^{-1}$. 
Then $\varphi$ is a covering
of $\mathrm{Cay}(\Gamma', \varphi(A))$ 
by $\mathrm{Cay}(\Gamma, A)$.
\end{lemma}

All constructions of perfect colorings 
of $H(n,q)$ 
in this section  are based 
on the following 
known and straightforward fact.
\begin{lemma} 
\label{colorincover}
 Suppose that $f$ is a perfect coloring 
 of a multigraph $H$ 
 with the quotient matrix $S$.
 If $\varphi$ 
 is a covering of  $H$ 
 by a multigraph $G$,
 then  $f \circ \varphi$ 
 is a perfect coloring of $G$
 with the quotient matrix $S$.
\end{lemma}

Given a multigraph $G$ 
and $u,v \in \mathbb{N}$, 
we denote by $G + u I$ 
the multigraph obtained from $G$ 
by adding $u$ loops to each vertex of $G$ 
and by $vG$ the multigraph 
in which every edge and loop of $G$ has 
the multiplicity $v$ times larger 
than in $G$.
(We make an agreement that each 
loop contributes $1$ to the degree
of the corresponding vertex.)
If $G$ is a Cayley multigraph 
$\mathrm{Cay}(\Gamma,A)$,
then the  multigraph 
$G + u I$ can be represented as 
$\mathrm{Cay}(\Gamma, 
A \cup \{ u \times \bar 0\})$,
where $ u \times \bar 0$  is the identity element $\bar 0$
of $\Gamma$ with multiplicity $ u $.
Similarly, the multigraph $ v G$ coincides with 
$\mathrm{Cay}(\Gamma, v A)$, 
where $ v A$ is the multiset obtained from $A$
by multiplying all multiplicities by $ v $.

The definitions 
of a perfect coloring 
and multigraphs $G + u I$ and $ v G$ 
imply the following.

\begin{lemma} \label{colorinplustimes}
\item \begin{enumerate}
\item If $f$ is a perfect coloring 
of a multigraph $G$ 
with the quotient matrix $S$, 
then $f$ is a perfect coloring 
of the multigraph $G + u I$ 
with the quotient matrix $S + u I$.
\item If $f$ is a perfect coloring of a multigraph $G$ with the quotient matrix $S$, 
then $f$ is a perfect coloring of the multigraph $ v G$ with the quotient matrix $ v S$.
\end{enumerate}
\end{lemma}

We consider a series of graphs
and multigraphs that can be covered
by Hamming graphs. 

\begin{proposition} \label{coverHam}
\item \begin{enumerate}
\item The Hamming graph $H(n+ u ,q)$ covers 
the multigraph $H(n,q) + u (q-1)I$.
\item The Hamming graph $H(v n,q)$ covers 
the multigraph $ v H(n,q)$.
\item The Hamming graph $H(n,pq)$ 
covers the multigraph $ pH(n,q) + n(p-1)I$.
\end{enumerate}

\end{proposition}

\begin{proof}
Let us prove  (1).  Consider the group homomorphism  
$\varphi:  \mathbb{Z}_q^{n+ u } \rightarrow \mathbb{Z}_q^n$ 
defined by the equation
$$
\varphi (x_1, \ldots, x_n , 
              \ldots,  x_{n+ u }) 
= (x_1, \ldots, x_n).
$$
It is straightforward to see the following:
if $\mathcal{I}(n+ u ,q)$ 
is the set of all elements of 
$\mathbb{Z}_q^{n+ u }$ 
at distance $1$ 
from the identity element, 
then the multiset 
$B: = \varphi(\mathcal{I}(n+ u ,q)) $ 
is equal to 
$ \mathcal{I}(n,q) \cup \{ u (q-1)\times \bar 0\}$.
Since $\mathrm{Cay}(\mathbb{Z}_q^n, B)$
is exactly the multigraph 
$H(n,q) + u (q-1)I$, 
Lemma~\ref{homocover} 
implies that $\varphi$ 
is a covering of 
$H(n,q) + u (q-1)I$ 
by $H(n+ u ,q)$.

The proofs of (2) and (3) are similar to (1).  To prove (2), we use the group homomorphism  
$$\varphi (x_1, \ldots, x_{ v n}) = (x_1 + \cdots + x_v, \ldots, x_{ v (n-1)+1} + \cdots + x_{ v n}),$$
where $+$ is the operation of $\mathbb{Z}_q$, and for (3) we utilize the homomorphism
$$
\varphi (x_1, \ldots, x_{n}) = 
(x_1, \ldots, x_n ) \bmod q; 
\quad x_{i}\in \mathbb{Z}_{pq}.
$$
\end{proof}

\begin{remark}
There are many other ways
to choose
homomorphisms $\varphi$ 
in the proof of 
Proposition~\ref{coverHam}. 
For instance, 
it can be checked that 
if $g^1, \ldots, g^n $
are arbitrary $v$-ary quasigroups 
of order $q$ such that
$g^i(0, \ldots, 0) = 0$ 
for all $i$ 
then the operation 
$$
\varphi  (x_1, \ldots, x_{ v n})= (g^1(x_1, \ldots, x_v), \ldots, g^n(x_{ v (n-1)+1}, \ldots, x_{ v n}))
$$
is a group homomorphism
with properties
required in
Proposition~\ref{coverHam}(2).
\end{remark}

Using Lemmas~\ref{colorincover}, \ref{colorinplustimes} 
and Proposition~\ref{coverHam}, 
we obtain 
the following constructions 
of perfect colorings 
in Hamming graphs.

\begin{theorem}[{\cite[Prop.\,33(ii)]{BRZ:CR}}%
\footnote{There is a misprint in the parameters of the resulting coloring in \cite[Prop.\,33(ii)]{BRZ:CR}}] \label{diagconstr}
For every perfect coloring 
of $H(n,q)$ 
with quotient matrix $S$
and every positive integer $ u $, 
there exists a perfect coloring 
of $H(n+ u , q)$ 
with quotient matrix 
$S + u (q-1)I$. 
In particular, 
if there is a perfect $(b,c)$-coloring
of $H(n,q)$, then there is a perfect $(b,c)$-coloring
of $H(n+ u , q)$, $ u =1,2,\ldots$.
\end{theorem}

\begin{theorem}[{\cite[Prop.\,33(iii)]{BRZ:CR}}]
\label{timestconstr}
For every perfect coloring of 
$H(n,q)$ with quotient matrix $S$
and every positive integer $ v $,
there exists a perfect coloring
of $H( v n, q)$ 
with quotient matrix $ v S$.
In particular, 
if there is a perfect $(b,c)$-coloring
of $H(n,q)$, then there is a perfect $( v b, v c)$-coloring
of $H( v n, q)$, $ v =1,2,\ldots$.
\end{theorem}

\begin{theorem}
\label{altimesconstr}
For every perfect coloring of 
$H(n,q)$ with quotient matrix $S$
and every positive integer $p$, 
there exists a perfect coloring 
of $H(n, pq)$ 
with quotient matrix 
$pS + n(p-1)I$.
In particular, 
if there is a perfect $(b,c)$-coloring
of $H(n,q)$, then there is a perfect $(pb,pc)$-coloring
of $H(n, pq)$, $p=1,2,\ldots$.
\end{theorem}

\subsection{Additive codes}\label{s:additive}
Under a \emph{code} in the Hamming graph $H(n,q)$, we mean an arbitrary nonempty subset of the vertex set of $H(n,q)$.
To each code in $H(n,q)$ we assign a $2$-coloring in which the set of vertices of the first color coincides with the code. We often identify codes with the corresponding $2$-colorings.

In this section we assume that $q=p^s$ for some prime $p$ and integer $s\ge 1$.
As the vertex set of the Hamming graph $H(n,q)$, we consider 
$$(\mathbb{Z}_p^s)^n =\{x=(x_1,\ldots,x_n)\mid x_i \in \mathbb{Z}_p^s \}.$$
Two vertices are adjacent in $H(n,q)$ if they differ in exactly one 
$\mathbb{Z}_p^s$-component.
A subset of $(\mathbb{Z}_p^s)^n$ is called an \emph{additive code}
if it is closed under the coordinate-wise addition.
Since $p$ is prime, an additive code is also a \emph{linear code} over $\mathbb{Z}_p$,
i.e., it is also closed with respect to the coordinate-wise multiplication by a constant.
As any linear code, such a code $C$ is the kernel of a vector-space homomorphism
represented by a $(sn-\dim C)\times sn$ matrix over $\mathbb{Z}_p$ called 
a \emph{check matrix} of the code.

\begin{theorem}\label{th:additive}
Let $C$ be an additive code in $(\mathbb{Z}_p^s)^n$ with an $m\times sn$ check matrix $H$ consisting 
of columns $h_{1,1}$, \ldots, $h_{1,s}$, $h_{2,1}$, \ldots, $h_{n,s}$. 
For $i\in\{1,\ldots,n\}$, 
denote by $V_i$ the vector subspace of $\mathbb{Z}_p^m$ spanned by 
$h_{i,1}$, \ldots, $h_{i,s}$. 
The following assertions are equivalent.
\begin{itemize}
 \item [\rm(i)] All the spaces $V_i$, $i=1,\ldots,n$,
 are of dimension $s$,
 and every non-zero element of
 $\mathbb{Z}_p^m$ belongs to exactly $c$ of them,
 $c=(p^s-1)n/(p^m-1)$.
 \item [\rm(ii)] The characteristic function 
 of $C$ in $(\mathbb{Z}_p^s)^n$ 
 is a perfect $2$-coloring of $H(n,q)$
 with quotient matrix
 $[[0,b],[c,b-c]]$, $b=(p^s-1)n$, $c=b/(p^m-1)$.
 \item [\rm(iii)] $C$ is an 
       OA$\left(p^{sn-m},n,p^s, \frac{p^{m-s}(p^s-1)n}{p^m-1}-1\right)$.
\end{itemize}
\end{theorem}
\begin{proof}
 (ii) implies (iii) by Proposition~\ref{corrimmune}. 
 As noted in~\cite{Pot:2012:color}, any OA$(N,n,q,t)$
 attaining the Bierbrauer--Friedman bound 
 $ N \ge q^n\left( 1-\frac{(q-1)n}{q(t+1)}  \right)$
 is a simple independent set whose characteristic function is a $(b,c)$-coloring,
 $b=(q-1)n$, $c=q(t+1)-(q-1)n$. So, (ii) follows from (iii).
 
 Let us show the equivalence of (i) and (ii). 
 Clearly, $C$ is an independent set 
 if and only if all $V_i$ are of maximal 
 dimension, $s$. 
 Assume this is the case.
 Consider a non-code vector $x$ from  $(\mathbb{Z}_p^s)^n$. 
 Denote by $h$ the \emph{syndrome} $Hx^{\mathrm T} $.
 If $h$ belongs to $V_i$ for some $i$, 
 then there is a unique vector $e$ with nonzeros only in the $i$th group of coordinates
 such that $He^{\mathrm T}=h$, so $x-e$ is a code neighbor of $x$. 
 We see that the number of $V_i$ the vector $h$ belongs to is exactly the number of code neighbors of $x$.
 Every non-code vector $x$ has exactly $c$ code neighbors if and only if every nonzero $h$ from 
 $\mathbb{Z}_p^m$ belongs to exactly $c$ spaces $V_i$, $i\in\{1,\ldots,n\}$.
 By numerical reasons, $c$ can only be $(p^s-1)n/(p^m-1)$.
\end{proof}

A collection of subspaces satisfying (i) is called a \emph{$c$-fold spread},
see~\cite[p.83]{Hirschfeld79}, where a construction of multifold spreads can be found.
The multifold spreads are a special case of subspace designs, namely,
the subspace designs of strength $1$, or $q$-ary $1$-$(n,k,\lambda)$ designs 
(in our notation, $k=m$ and $\lambda=c$).

\begin{Example}
Taking $p=s=2$, $m=3$, $n=7$, 
and all seven two-dimensional 
subspaces of $\mathbb{Z}_2^3$ as $V_i$, 
we obtain a $(21,3)$-coloring
of $H(7,4)$ and OA$(2048,7,4,5)$.
The last parameters 
(as well as the derived parameters OA$(512,6,4,4)$) 
occur 
as putative in~\cite[Table~12.3]{HSS:OA} 
(where $k:=n=7$, $t:=5$, index is $2048/4^5=2$).
However, the parameters 
(the length $7$, the dimension $3$ over $\mathbb{Z}_2$, 
and the minimum distance $d=6$) 
of the corresponding quaternary additive code
generated by $H$ (i.e., dual to $C$) 
were already discovered 
in~\cite{BlokBrow:2004},
and by the Delsarte theory~\cite{Delsarte:1973}, 
the strength of the dual orthogonal array is $t:=d-1$.
So, we cannot say that the OA$(2048,7,4,5)$
is new, but now we can treat it as a special 
case of a general construction,
in contrast to the computer-aided approach in~\cite{BlokBrow:2004}. 
The union of $3$ cosets of $C$ gives a perfect $(15,9)$-coloring
of $H(7,4)$, which reduces the upper bound in another line of the small-value table
(Appendix).
\end{Example}
 
It is possible to construct perfect colorings from additive codes 
that are not independent sets. This possibility needs further investigation,
and here we only show one example, whose parameters occur in Appendix.
 
\begin{Example}
Consider the check matrix
$$
H=\left(
\begin{array}{c@{\,}c@{~}c@{\,}c@{~}c@{\,}c@{~}c@{\,}c@{~}c@{\,}c@{~}c@{\,}c@{~}c@{\,}c@{~}c@{\,}c@{~}c@{\,}c@{~}c@{\,}c@{~}c@{\,}c@{~}c@{\,}c@{~}c@{\,}c}
0 & 0  &  0 & 0  &  0 & 1  &  0 & 1  &  1 & 0  &  1 & 0  &  0 & 1  &  0 & 1  &  1 & 0  &  0 & 1  &  0 & 1  &    1 & 0   \cr
0 & 1  &  0 & 1  &  1 & 0  &  1 & 0  &  0 & 0  &  0 & 0  &  0 & 1  &  1 & 0  &  0 & 1  &  1 & 1  &  1 & 1  &    1 & 0   \cr
1 & 0  &  1 & 0  &  0 & 0  &  0 & 0  &  0 & 1  &  0 & 1  &  1 & 0  &  0 & 1  &  0 & 1  &  1 & 0  &  1 & 0  &    1 & 0   
\end{array}
\right).
$$
over $\mathbb{Z}_2$. 
The kernel of this matrix is an additive code, 
whose characteristic function 
is a perfect $(35,5)$-coloring of $H(12,4)$ 
(a perfect $(35,5)$-coloring of $H(13,4)$ was constructed in
\cite{KroPot:multifold} as a special case of perfect multifold ball packing).
The proof is straightforward,
using the same syndrome approach as in Theorem~\ref{th:additive}.
The union of $3$ cosets gives a perfect $(20,15)$-coloring
and also contributes to the small-value table.
\end{Example}

\begin{remark}
Additive codes can be interpreted in the classical manner of coding theory, as codes over fields. 
To do so, we can treat $\mathbb{Z}_p^s$ as the finite field GF$(q)$, with properly defined multiplication.
As a code over GF$(q)$, an additive code is closed with respect to addition but not necessarily 
with respect to multiplication by constant, if $s>1$. 
In the case, $s=1$, the additive construction gives only colorings
obtained from $1$-perfect codes, see the next subsection.
\end{remark}


\subsection{MDS codes, $1$-perfect codes}
\label{perfcodes} \label{directconstrsec}

An \emph{MDS code} with distance $d$
 is a set of $q^{n-d+1}$ vertices in $H(n,q)$ 
such that the Hamming distance 
between any two different code vertices 
is not less than $d$.
Equivalently, a distance-$d$ MDS code is 
an orthogonal array OA$(q^{n-d+1},n,q,n-d+1)$.

A \emph{$\tau$-fold MDS code}
in $H(n,q)$ is a set of vertices that 
has exactly $\tau$ elements in every 
$1$-face (line), i.e., a simple OA$(\tau q^{n-1},n,q,n-1)$.
Some $\tau$-fold MDS codes
can be obtained as the union of $\tau$ 
disjoint copies of distance-$2$ MDS codes. 
The definition implies that any 
$\tau$-fold MDS code, $\tau\in\{1,\ldots,q-1\}$,  
corresponds to an
$(n(q-\tau), n \tau)$-coloring of $H(n,q)$;
by Theorem~\ref{timestconstr}, 
colorings with such parameters 
can be constructed from 
a $(q-\tau,\tau)$-coloring 
of the complete graph $H(1,q)$.

Perfect colorings corresponding 
to  
 multifold MDS codes (including distance-$2$ MDS codes)
have the second eigenvalue 
$\theta = -n$,
which is the smallest eigenvalue
of $H(n,q)$.

A \emph{$1$-perfect code} 
in  $H(n,q)$ is a set of vertices 
such that each radius-$1$ ball 
$B(x) = \{ y  \mid  d(x,y) \leq 1\}$ 
contains exactly one code vertex. 
The minimal Hamming distance between 
different
vertices of a nontrivial 
(with more than $1$ vertex) 
$1$-perfect code 
is equal to $3$.

It is well known that 
$1$-perfect codes
exist in $H(n,q)$ 
if $q = p^s$ 
for some prime $p$ and 
$n = \frac{q^r -1}{q-1}$
for some $r \in \mathbb{N}$ 
(e.g. $q$-ary Hamming codes, 
which can be treated as a special case 
of the additive codes from Section~\ref{s:additive}).
If $q$ is not a prime power, 
then the existence of $1$-perfect codes
in $H(n,q)$ 
is a long-standing open problem.
The definition implies 
that every $1$-perfect code 
in $H(n,q)$ is a $(n(q-1),1)$-coloring.

A \emph{$\tau$-fold $1$-perfect code} 
in $H(n,q)$ 
is a set of vertices 
such that each radius-$1$ ball 
contains exactly $\tau$ code vertices.
Some $\tau$-fold $1$-perfect codes
 can be obtained as 
the union of $\tau$ disjoint copies 
of $1$-perfect codes. 
$\tau$-Fold $1$-perfect codes 
correspond to $(n(q-1)-\tau+1, \tau)$-colorings 
of $H(n,q)$; they exist if 
(but not necessarily ``only if'') 
$1$-perfect codes exist.
The second eigenvalue 
of $1$-perfect codes 
and multifold $1$-perfect codes is $-1$.

In constructions, we will use 
the following correspondence 
between $1$-perfect codes 
in the Hamming graph $H(q+1, q)$ 
and certain  MDS codes in $H(q,q)$. 

\begin{proposition} \label{Hqqdecomp}
If there exists a $1$-perfect code in  
$H(q+1, q)$, 
then there exists a partition
of the vertex set of $H(q,q)$ 
into  distance-$2$ MDS codes 
$M^0$, \ldots, $M^{q-1}$ 
such that each $M^i$ 
can be decomposed 
into the union of 
$q$ disjoint distance-$3$ 
MDS codes.  
\end{proposition}

\begin{proof}
Let $C$ be a $1$-perfect code
(equivalently, a distance-$3$ MDS code)
in $H(q+1,q)$,
and let $C_j$ be the subset
of $C$ consisting of all
codewords with last symbol $j$, 
$j=0,\ldots,q-1$.
Denote by $M$, 
$L_0$, \ldots, $L_{q-1}$
the last-coordinate 
projections of $C$, 
$C_0$, \ldots, $C_{q-1}$, respectively.
By definition, 
$M$ is a distance-$2$
MDS code, and $L_j$, $j=0,\ldots,q-1$,
are distance-$3$ MDS codes;
moreover, $M=\cup_j L_j$.
It follows that the codes
$M^i=M + (i,0,...,0)$, $i=0,\ldots,q-1$,
form a required partition.
\end{proof}


\subsection{The invasion construction} \label{s:invasion}
Here, we describe a construction 
of  perfect $2$-colorings 
in Hamming graphs 
based on a set of $2$-colorings 
and a perfect coloring 
into an arbitrary number of colors. 
We  will use it as a step of 
splitting constructions in Sections~\ref{s:split1} and~\ref{s:split2}.

Let $f$ be a $k$-coloring of  $H(n,q)$ 
and let $g_1$, \ldots, $g_k$ 
be $2$-colorings of $H(m,q)$ with same sets of colors.
Define the \emph{invasion}  
$h = f \times (g_1, \ldots, g_k)$ 
of the coloring $f$ by colorings 
$g_1$, \ldots, $g_k$ to be a 
$2$-coloring of $H(n+m,q)$ such that
$$h(x,y) = g_{f(x)} (y), ~~~ x \in \mathbb{Z}_q^n, ~~y \in \mathbb{Z}_q^m.$$

Roughly speaking, the invasion coloring $h$ 
of $H(n+m,q)$ is obtained  
by replacing each vertex of color $i$ 
in the  coloring  $f$ of $H(n,q)$ 
by the graph $H(m,q)$ colored 
with the coloring $g_i$.

Assume that $M^1$, \ldots, $M^q$ is 
a partition of the vertex set of $H(m,q)$ 
into $q$ pairwise  disjoint distance-$2$ 
MDS codes. 
Given $\tau$, $0 \leq \tau \leq q$, 
define $G^i_\tau$ to be the $2$-coloring
corresponding to the $\tau$-fold MDS code 
$\bigcup\limits_{j=i}^{i+\tau-1} M^j$ 
(index $j$ goes cyclically modulo $q$) 
in $H(m,q)$ if $m \geq 1$. 
In the case $m = 0$ (when $H(0,q)$ is the $1$-vertex graph), 
let  $G^i_\tau$ be the coloring 
of a vertex into the first color 
if $1 \leq i \leq \tau$,
and the coloring of a vertex 
into the second color  if
$\tau+1 \leq i \leq q$.

\begin{proposition} \label{invasion}
{1.} Let $f$ be a perfect $(q+1)$-coloring 
of $H(n,q)$ with the quotient matrix
$$ \left(\begin{array}{cccc}
\alpha' &  \cdots & \alpha & \beta \\
\vdots  & \ddots &  \vdots & \vdots \\
\alpha &  \cdots & \alpha' & \beta \\
\gamma & \cdots & \gamma & \delta 
\end{array} \right).$$
If $ 0 \leq  \gamma - \alpha = m$ then 
for each $\tau = 0, \ldots, q$, $l = 1, 2$ such that $\tau (l-1) + (q-\tau) (2-l) \neq 0$, 
the invasions 
$h_l = f \times (G^1_\tau, \ldots, G^q_\tau, \mathbf{l})$,
where $\mathbf{l}$ 
is the solid  coloring of $H(m,q)$ 
into the color $l$, 
are $(b_l,c_l)$-colorings 
of $H(n+m,q)$ 
with  $b_l = \gamma (q-\tau) + \beta (l-1)$ 
and $c_l = \gamma \tau + \beta (2-l)$.

{2.} Let $f$ be a perfect $2q$-coloring 
of $H(n,q)$ with the quotient matrix
$$ \left(\begin{array}{cccccc}
\alpha &  \cdots & \alpha & \beta & \cdots & \beta \\
\vdots  & \ddots &  \vdots & \vdots & \ddots & \vdots \\
\alpha &  \cdots & \alpha & \beta & \cdots & \beta\\
\gamma & \cdots & \gamma & \delta & \cdots & \delta \\
\vdots  & \ddots &  \vdots & \vdots & \ddots & \vdots \\
\gamma & \cdots & \gamma & \delta & \cdots & \delta \\
\end{array} \right).$$
If $0 \leq \gamma - \alpha = \beta - \delta = m$,
then for each $\tau_1, \tau_2 = 0, \ldots, q$, $\tau_1 + \tau_2 \neq 0, 2q$ 
the invasion 
$h = f \times (G^1_{\tau_1}, \ldots, G^q_{\tau_1}, G^1_{\tau_2}, \ldots, G^q_{\tau_2})$ 
is a $(b,c)$-coloring 
of $H(n+m,q)$ 
with 
$b = q(\gamma + \beta) -\gamma \tau_1 - \beta \tau_2$ 
and $c = \gamma \tau_1 + \beta \tau_2$.
\end{proposition}

\begin{proof}
1. By construction,  
the invasion 
$h_l = f \times (g_1, \ldots, g_k)$  
is a coloring of $H(n+m,q)$ 
into colors $1$ and $2$. 
Each vertex $z \in H(n+m,q)$ 
is considered as an ordered pair 
of vertices $(x,y)$,  
where $x$ is a vertex of  $ H(n,q)$ 
and $y$ is a vertex of $ H(m,q)$. 
Therefore, the number of vertices 
of color $1$ (color $2$) 
in the coloring $h_l$ 
adjacent to the vertex $z = (x,y)$ 
is equal to  the number of vertices 
of color $1$ (color $2$) adjacent to $y$
in the coloring $G^{f(x)}_\tau$ 
or in the coloring $\mathbf{l}$ 
plus the number of vertices $h(x',y)$
of color $1$ (color $2$), 
where $x'$ runs over all 
the neighbors of $x$.

Recall 
that for each 
$i = 1, \ldots, q$,
the coloring $G^i_\tau$
is a $(m(q-\tau),m \tau)$-coloring of  $H(m,q)$.

Suppose that $h_l(z)=1$, $z=(x,y)$.
If $1 \leq f(x) \leq q$,
then the number of vertices 
of $h_l^{-1}(2)$ adjacent to $z$  
is
$m(q-\tau) + \alpha (q-\tau) + \beta (l-1)$.  
In the case $ f(x) = q+1$,
we have $l=1$ and
the number of the vertices 
of $h_1^{-1}(2)$ adjacent to $z$ 
is equal to $\gamma (q-\tau)$. 
The condition $m = \gamma - \alpha$ 
implies that the equality 
$m(q-\tau) + \alpha(q-\tau) = \gamma (q-\tau)$ 
holds for the coloring $h_1$. 

Suppose now that 
$h_l(z)=2$, $z=(x,y)$. 
If $1 \leq f(x) \leq q$,
then the number of the vertices 
of $h_l^{-1}(1)$ adjacent to $z$ 
is equal to $ m \tau + \alpha \tau + \beta (2-l)$.  
In the case $ f(x) = q+1$,
we have $l=2$ and
the number of the vertices 
of $h_2^{-1}(1)$ adjacent to $z$
in the coloring $h_2$ 
is equal to $\gamma \tau$. 
Again, the condition
$m = \gamma - \alpha$ 
implies that 
the equality 
$ m \tau + \alpha \tau = \gamma \tau $ 
holds for the coloring $h_2$. 

Thus, the invasions $h_l$ 
are perfect $(b_l, c_l)$-colorings
of $H(n+m,q)$ with 
$b_l = \gamma (q-\tau) + \beta (l-1)$ and 
$c_l = \gamma \tau + \beta (2-l)$.

2. As before, let $z = (x,y)$ 
be a vertex of  $H(n+m,q)$. 

Suppose $h(z)=1$. 
By the definition of $h$, 
if $1 \leq f(x) \leq q$ 
then the number of neighbors
of $z$ in $h^{-1}(2)$ 
is equal to 
$m(q-\tau_1) + \alpha (q-\tau_1) + \beta (q-\tau_2)$.
In the case $q+1 \leq  f(x) \leq 2q$
the number of vertices of color 2
adjacent to $z$ is equal to
$m (q-\tau_2) + \gamma (q-\tau_1) + \delta(q-\tau_2)$. 
The condition 
$m = \gamma - \alpha = \beta - \delta$ 
implies that the equality 
$m(q-\tau_1) + \alpha (q-\tau_1) + \beta (q-\tau_2) = m (q-\tau_2) + \gamma (q-\tau_1) + \delta(q-\tau_2)$ 
holds for  all $\tau_1$ and $\tau_2$. 

Now suppose $h(z)=2$.
If $1 \leq f(x) \leq q$ 
then the number 
of neighbors of $z$ 
in $h_l^{-1}(1)$ 
is equal to 
$m\tau_1 + \alpha \tau_1 + \beta \tau_2$.  
In the case $q+1 \leq f(x) \leq 2q+1$,
the number of 
vertices of color 1 adjacent to $z$ 
is equal to 
$m\tau_2 + \gamma \tau_1 + \delta \tau_2$. 
Again, the condition 
$m = \gamma - \alpha = \beta - \delta$ 
implies that the equality 
$m\tau_1 + \alpha \tau_1 + \beta \tau_2 
 = m \tau_2 + \gamma \tau_1 + \delta \tau_2$ 
holds for  all $\tau_1$ and $\tau_2$. 

Thus, the invasion $h$ is a perfect 
$(b, c)$-coloring of $H(n+m,q)$ 
with 
$b = \gamma (q-\tau_1) + \beta (q-\tau_2)$ 
and $c = \gamma \tau_1 + \beta \tau_2$.
\end{proof}

\subsection{Splitting construction I}
\label{flaasssec}\label{s:split1}

In this and next sections,
we provide  
constructions
that enable us to obtain 
perfect $2$-colorings 
with new proportions of colors.
We call them 
splitting constructions.
The main idea 
is to start with
multiplying the parameters of
a $2$-coloring by a factor $q$,
as in Theorem~\ref{timestconstr};
after this, with the help of
$1$-perfect code in $H(q+1,q)$,
one of the colors 
(say, the second one) 
is split into
$q$ ``equivalent'' colors; 
and finally, 
after some intermediate step,
$\tau$ of these colors are
unified with the first one,
while the remaining $q-\tau$ 
colors form the new second color. 
So, the ratio of the colors 
changes as follows:
$$
c:b=qc:qb
\to
qc:b:\cdots:b \to qc+ \tau b:(q- \tau )b.
$$
The constructions generalize
the main construction 
in \cite{FDF:PerfCol}
from the case $q=2$ 
to the case of an arbitrary $q$ 
such that there exists 
a $1$-perfect code 
in $H(q+1,q)$.

We divide their proofs  
into several lemmas.
Firstly,  for certain $q$ 
and for every perfect $(b,c)$-coloring  of $H(n,q)$, 
we construct an appropriate perfect $2q$-coloring.

\begin{lemma} \label{2colorsinq}
Let $f$ be a perfect coloring 
in $H(n,q)$ 
with the quotient matrix 
$ \left( \begin{array}{cc} a & b 
  \\ c & d \end{array} \right)$.  
If there exists
a $1$-perfect code 
in $H(q+1,q)$,
then there exists
a perfect $2q$-coloring $g$
of $H(qn,q)$
 with the quotient matrix
$$
T = \left( \begin{array}{cccccc}
a & \cdots & a & b & \cdots & b \\
\vdots& \ddots & \vdots & \vdots & \ddots & \vdots \\
a & \cdots & a & b & \cdots & b \\
c & \cdots & c & d & \cdots & d \\
\vdots& \ddots & \vdots & \vdots & \ddots & \vdots \\
c & \cdots & c & d & \cdots & d \\
\end{array} \right). 
$$
\end{lemma}

\begin{proof}
Let $M^1$, \ldots, $M^q$ 
be a partition 
of the vertex set of  $H(q,q)$ 
into $q$ distance-$2$ MDS codes
such that each $M^i$ 
is partitioned
into distance-$3$ MDS 
codes $L^i_1$, \ldots, $L^i_q$. 
The existence of such a decomposition, 
provided the existence a $1$-perfect code 
in $H(q+1,q)$, 
is guaranteed by 
Proposition~\ref{Hqqdecomp}.

Define a $q$-ary quasigroup $h$ 
of order $q$ 
on the vertex set of $H(q,q)$ as 
$h(x_1, \ldots, x_q) = i $ 
if and only if the vertex 
$(x_1, \ldots, x_q) $ 
belongs to $ M^i$. 

Recall that every vertex $y = (y_1, \ldots, y_{qn})$ of $H(qn,q)$ 
can be considered as a tuple 
$y = (y^1, \ldots, y^n)$ 
of $n$ vertices 
$y^i = (y_{(i-1)q+1}, \ldots, y_{iq})$ 
of $H(q,q)$. 
For shortness, 
let $X_y = (x_1, \ldots, x_n)$ 
be a vertex of $H(n,q)$ 
with $x_i  = h(y^i)$ 
and $J_y = (j_1, \ldots, j_n)$ 
be a vertex of $H(n,q)$ 
such that $j_i$ is defined
by the index of the distance-$3$  MDS code  
containing the vertex $y^i$:  
$y^i \in L^{h(y^i)}_{j_i}$.  
For an arbitrary $n$-ary quasigroup $R$ 
of order $q$ on the vertex set of $H(n,q)$, 
we define a coloring $g$ of $H(qn,q)$ as
$$g(y)  =  q(f(X_y) - 1) + R(J_y). $$
So $g$ is a $2q$-coloring of $H(qn,q)$ 
into colors 
$\{ 1, \ldots, q, q+1, \ldots, 2q\}$.

Let us prove that $g$
is a perfect coloring 
with the quotient matrix $T$. 
Let $\mathcal{Z}^i_\alpha$
be the set of vertices $z$ adjacent to $y$ 
such  that  $z$ differs from $y$ 
in some component of $y^i$ from $H(q,q)$ with $\alpha = h(z^i) \neq h(y^i)$. Given a vertex $y$ with $f(X_y) = 1$, there are exactly $a$ sets  $\mathcal{Z}^i_\alpha$ such that for all $z \in \mathcal{Z}^i_\alpha$ we have $f(X_z) = 1$ and  $b$ sets  $\mathcal{Z}^i_\alpha$ with $f(X_z) = 2$ for all $z \in \mathcal{Z}^i_\alpha$. Similarly, if $f(X_y) = 2$ then we have $c$ sets  $\mathcal{Z}^i_\alpha$ such that for all $z \in \mathcal{Z}^i_\alpha$ it holds $f(X_z) = 1$ and  $d$ sets  $\mathcal{Z}^i_\alpha$ with $f(X_z) = 2$ for all $z \in \mathcal{Z}^i_\alpha$.

 Note that 
 for each $i$ and $\alpha$ 
 the cardinality of the set 
 $\mathcal{Z}^i_\alpha$ 
 is equal to $q$ 
 and all these $q$ vertices $z$ 
 are colored by $g$ 
 into $q$ different colors. 
 Indeed, their components $z^i$ 
 belong to different 
 distance-$3$ MDS codes 
 $L^\alpha_{j_{i}}$ 
  (otherwise we have a contradiction
  with the minimal distance 
  in $L^\alpha_{j_{i}}$), 
  and all other components  
  $z^k$ of 
  $z \in \mathcal{Z}^i_\alpha$ 
  coincide and belong 
  to the same codes 
  $L^{h(z^k)}_{j_k}$. 
  Since the quasigroup  $R$ 
  takes all different values 
  on the set of $q$ vectors $J_z$
  different in one position, 
  vertices 
  $z \in \mathcal{Z}^i_\alpha$
  are colored by $g$ into all 
  $q$ possible colors.   
 
 Therefore, each vertex $y$ 
 with  $g(y) \in \{ 1, \ldots, q\}$ 
 is adjacent to exactly $a$ 
 vertices of each of the colors 
 $1, \ldots, q$ 
 and is adjacent to $b$ 
 vertices of each 
 of the colors $q+1, \ldots, 2q$ 
 in the coloring $g$. 
 The same is true 
 for vertices $y$ of colors 
 $g(y) \in \{ q+1, \ldots, 2q\}$.
\end{proof}

In case when one of the colors
of a perfect $2$-coloring $f$ 
can be divided into 
$k$-dimensional faces,
a similar method enables us
to construct  the following 
perfect $(q+1)$-colorings.

\begin{lemma} \label{2colorswithkinq}
Let $f$ be a perfect coloring 
in $H(n,q)$ 
with the quotient matrix 
$ \left( \begin{array}{cc} a & b 
 \\ c & d \end{array} \right)$
such that the set of vertices 
of the first color can be 
partitioned into  $k$-faces.  If there exists a $1$-perfect code in   $H(q+1,q)$,
then there exist perfect colorings $g'$ and $g''$ in $q+1$ colors in $H(qn,q)$ 
with the quotient matrices
$$T' = \left( \begin{array}{cccc}
a - k(q-1) &   \cdots & a+k & qb  \\
\vdots & \ddots & \vdots & \vdots  \\
a + k  &  \cdots & a-k(q-1) & qb  \\
c & \cdots & c & qd   \\
\end{array} \right) ; ~~
T'' = \left( \begin{array}{cccc}
a + k(q-1)^2 &   \cdots & a-k(q-1) & qb  \\
\vdots & \ddots & \vdots & \vdots  \\
a - k(q-1)  &  \cdots & a+k(q-1)^2 & qb  \\
c & \cdots & c & qd   \\
\end{array} \right).
 $$
\end{lemma}

\begin{proof}
In the proof of this lemma, we use the same notations as in Lemma~\ref{2colorsinq}.
In addition, if a vertex $x$ of $H(n,q)$ is colored with the first color in the coloring $f$, then let $I_x = \{ i_1, \ldots, i_k\}$ be the set of $k$ free directions in the $k$-dimensional face $\Gamma$ containing the vertex $x$ in the demanded decomposition of this color into faces. We will say that $I_x$ is the set of special directions for the vertex $x$.

 For an arbitrary $n$-ary quasigroup $R'$ of order $q$ on vertices of $H(n,q)$,  define a coloring $g'$ of $H(qn,q)$ as
$$g'(y) = g'(y^1, \ldots, y^n) = \left\{   \begin{array}{l} 
R'(j_1, \ldots, j_n) \mbox{ if } f(X_y) = 1 , \mbox{ where } j_i = h(y^i) \mbox{ for } i \in I_{X_y} \mbox{ and }  y^i \in L^{h(y^i)}_{j_i} \mbox{ otherwise}; \\
q + 1 \mbox{ if } f(X_y) = 2 ,
 \end{array}  \right. $$
 and for  a $(n-k)$-ary quasigroup $R''$  of order $q$ define a coloring $g''$ of $H(qn,q)$ as
$$g''(y) = g''(y^1, \ldots, y^n) = \left\{   \begin{array}{l} 
R''(j_{i_1}, \ldots, j_{i_{n-k}}) \mbox{ if } f(X_y) = 1, \mbox{ where all }  i_{l} \notin I_{X_y} \mbox{ and }  y^{i_l} \in L^{h(y^{i_l})}_{j_{i_l}} ;\\
q + 1 \mbox{ if } f(X_y) = 2.
 \end{array}  \right. $$
Let us prove that $g'$ and $g''$ are perfect $(q+1)$-colorings with the quotient matrices $T'$ and $T''$ respectively. 

As before, let $\mathcal{Z}^i_\alpha$ be a set of $q$ vertices $z$ adjacent to $y$ such  that  $z$ differs from $y$ in some component of $y^i$ from $H(q,q)$ with $\alpha = h(z^i) \neq h(y^i)$. 

1. We firstly prove that $g'$ is a perfect coloring. 

Assume that a vertex $y$ is colored into one of the colors $\{ 1, \ldots, q\}$ by the coloring $g'$. Acting similar to the proof of Lemma~\ref{2colorsinq}, we see that all $q$ vertices $z$ from the set $\mathcal{Z}_\alpha^i$ with $i \notin I_{X_y}$ and $f(X_z) =1$ are colored into all different colors. By definitions, all $qb$ vertices $z$  from sets $\mathcal{Z}_\alpha^i$ with $f(X_z) =2$  are different from $y$ in components $y^i$, where $i$ is not a special direction ($i \notin I_{X_y}$), so all these vertices are colored into color $q+1$ in the coloring $g'$.

Consider the set of vertices  $\mathcal{Z}_\alpha^i$ adjacent to the vertex $y$ along a special direction $i \in I_{X_y}$. By definition of the coloring $g'$, all $q$ vertices $z$ from $\mathcal{Z}_\alpha^i$ are colored into one color that is different from the color of the vertex $y$. Moreover, while $\alpha$ runs over all $q-1$ values different from $h(y^i)$, the colors of the vertices in sets $\mathcal{Z}_\alpha^i$ run over all $q-1$ colors different from the color of $y$.

The above reasoning implies that in the coloring $g'$ every vertex $y$ of a color from the set $\{ 1, \ldots, q \}$ is adjacent to $a - k(q-1)$ vertices of the same color, to $qb$ vertices of the color $q+1$ and to $a+k$ vertices of each of the remaining colors.

If the vertex $y$ is colored into the color $q+1$ by $g'$ then all vertices $z$ from sets $\mathcal{Z}_\alpha^i$  are either  colored into the color $q+1$ (when $f(X_z) = 2$) or colored into all $q$ colors $\{ 1, \ldots, q\}$ (when $f(X_z) = 1$), that gives us the required parameters.

2. Let us prove that $g''$ is a perfect coloring.

The coloring $g''$ is different from $g'$ only in neighborhoods along special directions $i \in I_{X_y}$ of vertices $y$  for which $f(X_y) = 1$. In that case, all $q$ vertices $z$ from the sets $\mathcal{Z}_\alpha^i$ are colored into the same color as the vertex $y$. 

It implies that in the coloring $g''$ every vertex $y$ of a color from the set $\{ 1, \ldots, q \}$ is adjacent to $a + k(q-1)^2$ vertices of the same color, to $qb$ vertices of the color $q+1$ and to $a-k(q-1)$ vertices of each of the remaining colors.

If the vertex $y$ 
is colored into the color $q+1$
by the coloring $g''$,
then we have the same coloring 
of its neighborhood 
as for the coloring $g'$.
\end{proof}

\begin{remark}
In constructions of colorings in Lemmas~\ref{2colorsinq} and~\ref{2colorswithkinq} we can take different quasigroups $R$, $R'$ and $R''$ for suitable sets of vertices that gives many nonequivalent colorings with the same parameters.
\end{remark}

Now we are ready to prove 
the main constructions. 
We start with a construction 
working for  perfect colorings
with the non-positive 
second eigenvalue.

\begin{theorem} \label{Flaassstand}
Assume that there is
a $1$-perfect code in $H(q+1, q)$. 
 If $f$ is a $(b,c)$-coloring
 in  $H(n,q)$ 
 with the second eigenvalue
 $\theta \leq 0$, 
 then for all
 $\tau_1, \tau_2 = 0, \ldots, q $,
 $\tau_1 + \tau_2 \neq 0, 2q$ 
 there exists  a
 $(q(b+c)- (c\tau_1 + b\tau_2), 
   c\tau_1 + b\tau_2)$-coloring $F$
 of  $H(qn - \theta, q)$.
 Moreover, $F$ has the same second 
 eigenvalue $\theta$.
\end{theorem}

\begin{proof}

Since there is a $1$-perfect code in $H(q+1,q)$, 
by Lemma~\ref{2colorsinq} 
there is a perfect 
$2q$-coloring $g$ 
in the Hamming graph $H(qn,q)$ 
with the quotient matrix 
$$
T = \left( \begin{array}{cccccc}
a & \cdots & a & b & \cdots & b \\
\vdots& \ddots & \vdots & \vdots & \ddots & \vdots \\
a & \cdots & a & b & \cdots & b \\
c & \cdots & c & d & \cdots & d \\
\vdots& \ddots & \vdots & \vdots & \ddots & \vdots \\
c & \cdots & c & d & \cdots & d \\
\end{array} \right). 
$$
Let 
$m = c - a = b - d = -\theta$. 
By the hypothesis 
of the theorem, 
we have $m \geq 0$. 
With the help of
Proposition~\ref{invasion}(2), 
we obtain that for each 
$\tau_1, \tau_2 = 0, \ldots, q$ 
there exists a 
$(\tilde{b}, \tilde{c})$-coloring
$F$ of $H(qn+m,q)$ with 
$\tilde{b} = q(b+c) -  c \tau_1 - b\tau_2$  
and $\tilde{c} = c \tau_1 + b \tau_2$.
 
The second eigenvalue 
of the perfect coloring $F$ is
$$
\tilde{\theta} = (q-1)(qn - \theta) - \tilde{b} - \tilde{c} = (q-1)(qn - \theta) - q (b+c) = (q-1)(qn - \theta) - q ((q-1)n - \theta) = \theta.
$$
\end{proof}

The next theorem is a more general
variant of the construction above,
applicable when the first color 
of a perfect coloring $f$ 
can be divided into 
$k$-dimensional faces.

\begin{theorem} \label{Flaassimprov}
Assume that there is
a $1$-perfect code in $H(q+1, q)$. 
Let $f$ be a $(b,c)$-coloring of $H(n,q)$ 
with the second eigenvalue $\theta$.
Suppose that the vertex set 
of the first color of $f$ 
can be partitioned into $k$-faces.
Then the following holds:
\begin{enumerate}
\item If $\theta + k \leq 0$, 
then for all 
$ \tau = 1, \ldots, q$ 
there is a 
$(q (b+c)- \tau c, \tau c)$-coloring $F$ 
of  $H(q n - \theta - k, q)$. 
The second eigenvalue of $F$ is
$\theta - k(q-1)$.
\item If $\theta\leq k(q-1)$,
then for all $ \tau = 1, \ldots, q$
there is a 
$(q(b+c)- \tau c, \tau c)$-coloring  $F$
of $H(qn - \theta + k(q-1), q)$. 
The second eigenvalue of $F$ is
$\theta + k(q-1)^2$.
\end{enumerate}
\end{theorem}

\begin{proof}
1. Since there is a $1$-perfect code 
in $H(q+1,q)$, 
by Lemma~\ref{2colorswithkinq} 
there is a perfect 
$(q+1)$-coloring $g'$ 
in $H(qn,q)$ with 
the quotient matrix 
$$
T' = \left( \begin{array}{cccc}
a - k(q-1) &   \cdots & a+k & qb  \\
\vdots & \ddots & \vdots & \vdots  \\
a + k  &  \cdots & a-k(q-1) & qb  \\
c & \cdots & c & qd   \\
\end{array} \right). 
$$
Let $m$ be 
$ c - a - k = -\theta - k \geq 0$. 
With the help of 
Proposition~\ref{invasion}(1),
we obtain that for each 
$\tau = 0, \ldots, q$ 
there exists a 
$(\tilde{b}, \tilde{c})$-coloring $F$ 
of $H(qn+m,q)$ with 
$\tilde{b} = q(b+c) -  c \tau $ 
and $\tilde{c} = c \tau $.

The second eigenvalue 
of the perfect coloring $F$ is
$$
\tilde{\theta} 
= 
(q-1)(qn - \theta -k) - \tilde{b} - \tilde{c} 
= 
(q-1)(qn - \theta -k) - q (b+c) 
= \theta - k(q-1).
$$

2. Since there is a $1$-perfect code 
in $H(q+1,q)$, 
by Lemma~\ref{2colorswithkinq} 
there exists a perfect 
$(q+1)$-coloring $g''$ 
of $H(qn,q)$ with the quotient matrix 
$$
T'' = \left( \begin{array}{cccc}
a + k(q-1)^2 &   \cdots & a-k(q-1) & qb  \\
\vdots & \ddots & \vdots & \vdots  \\
a - k(q-1)  &  \cdots & a+k(q-1)^2 & qb  \\
c & \cdots & c & qd   \\
\end{array} \right). 
$$
Let $m$ be
$ c - a + k(q-1) = -\theta + k(q-1) \geq 0$. 
With the help of 
Proposition~\ref{invasion}(1), 
we obtain that for each 
$\tau = 0, \ldots, q$ 
there exists a 
$(\tilde{b}, \tilde{c})$-coloring 
$F$ of $H(qn+m,q)$ with 
$\tilde{b} = q(b+c) -  c \tau$  
and $\tilde{c} = c \tau$.
 
The second eigenvalue of $F$ is
$$
\tilde{\theta}
= 
(q-1)(qn - \theta +k(q-1)) 
     - \tilde{b} - \tilde{c} 
= \theta + k(q-1)^2.
$$
\end{proof}

Note that for the perfect coloring  $F$ constructed in Theorem~\ref{Flaassimprov}(2) the vertex set of the first color can be partitioned into $kq$-dimensional faces. Moreover, for the second eigenvalue of the new coloring we have the inequality $\theta + k(q-1)^2 \leq kq(q-1)$.   So we can iteratively apply Theorem~\ref{Flaassimprov}(2) to the resulting coloring $F$.

\begin{corollary}
Assume that there is
a $1$-perfect code in $H(q+1, q)$. 
Let $f$ be a $(b,c)$-coloring of $H(n,q)$ 
with the second eigenvalue $\theta$.
If the vertex set 
of the first color of $f$ 
can be partitioned into $k$-faces,
then the following holds. If $\theta\leq k(q-1)$,
then for all $\tau_1, \ldots, \tau_r = 1, \ldots, q$
there is a 
$(q^r(b+c)- \tau_1 \cdots \tau_r c, \tau_1 \cdots \tau_r c)$-coloring  
of $H(q^r n + (k(q-1) - \theta) \frac{q^r - 1}{q-1}, q)$ 
with the second eigenvalue 
$\theta + k(q-1) (q^r -1)$.
\end{corollary}


\subsection{Splitting construction II}\label{s:split2}

We conclude this section 
with one more splitting construction
based on the covering 
of the graph $pH(n,q) +n(p-1)I$ by  $H(n,pq)$.

\begin{theorem} \label{splitconstr}
Let $f$ be a $(q^2-1,1)$-coloring 
of $H(q+1,q)$ 
corresponding to a 
$1$-perfect code 
such that the vertex set 
of the second (non-code)
color can be partitioned
into lines.
Then for every 
$p \in \mathbb{N}$ 
and for each 
$\tau$, $0 \leq \tau \leq p-1$
there exists a 
$((q^2-1)(p- \tau ) ,  
 (q^2-1) \tau  + p)$-coloring 
of  $H(q+1,pq)$ 
(with the second 
eigenvalue $q(p-1)-1$).
\end{theorem}

\begin{proof}
For a vertex
$x \in f^{-1}(2)\subset \mathbb{Z}_q^{q+1}$,
let $i_x$ denote 
the direction of the line 
containing this vertex 
in the demanded partition. 
We will say that $i_x$ 
is the special direction 
for the vertex $x$.
For a vertex $y$ 
of $H(q+1, pq)$, 
we denote by $X_y$ 
the vertex of $H(q+1,q)$ 
defined as 
$X_y \equiv y \mod q$
entrywise, and by $X'_y$,
the vertex of $H(q+1,p)$
equal to 
$(\lfloor \frac{y_1}{q}
\rfloor, \ldots, \lfloor
\frac{y_{q+1}}{q}  
\rfloor)$.

Let $g$ be a 
$(q (p- \tau ),q \tau )$-perfect coloring 
in  $H(q,p)$ 
corresponding to a 
$\tau$-fold MDS code 
in this graph.  
   Define a coloring $h$
   in  $H(q+1, pq)$ as follows
$$
h(y) = 
h(y^1, \ldots, y^{q+1}) =
\left\{   
\begin{array}{l} 
1 \mbox{ if } f(X_y) = 1; 
\\
g(x'_{i_1}, \ldots, x'_{i_{q}})\mbox{ if }  f(X_y) = 2, \mbox{ where all } i_l \neq i_{X_y},
 \end{array}  \right. $$
  where  $(x'_1, \ldots, x'_{q+1}) = X'_y $.
  Let us show that 
  the coloring $h$
  is a perfect coloring 
  of $H(q+1, pq)$ 
  with demanded parameters.
 
1. Let $y$ be a vertex
of $H(q+1,pq)$ 
colored by $h$ 
into the first color.
Let us count the number
of  $2$-colored vertices 
$z$ adjacent to $y$.
Denote 
$\mathcal{Z}_i^\alpha$ 
to be the set 
of $p$ vertices 
adjacent to $y$ 
such that 
$z_i  \equiv \alpha \not\equiv y_i \mod q$.

Assume that $f(X_y) = 1$.
Note that there are no
adjacent vertices $z$ 
colored by $h$ 
into second color 
such that $f(X_z) = 1$.
For each vertex $y$
there are exactly
$q^2-1$ sets 
$\mathcal{Z}_i^{\alpha}$
and, because $i$ is not
a special direction
for all $X_z$, 
$z \in \mathcal{Z}_i^\alpha$, 
each set
$\mathcal{Z}_i^{\alpha}$ 
contains exactly $\tau$
vertices colored by $h$
into the first color 
and $p- \tau $ vertices 
colored into the second color. 
So each such vertex $y$ 
is adjacent to
$(q^2-1) (p- \tau )$ 
vertices of the second color. 

Suppose now that 
$f(X_y) = 2$. 
By the definition 
of the coloring $h$, 
there are $q(p- \tau )$ vertices
$z$ adjacent to $y$ 
such that  
$y_i \equiv z_i \mod q$ 
for all $i$
and $h(z) =2$. 
As before, 
for each vertex $y$ 
there are exactly $q^2-1$
sets 
$\mathcal{Z}_i^{\alpha}$.
If $i = i_{X_y}$ 
is a special direction 
for the vertex $X_y$ 
in $H(q+1,q)$,
then for all 
$\alpha \not\equiv y_i \mod q$ 
all vertices $z$ 
from the sets 
$\mathcal{Z}_i^{\alpha}$
are colored 
into the first color by $h$. 
On the other hand, 
if $i \neq i_{X_y}$
then each of $q^2 - q -1$ sets $\mathcal{Z}_i^{\alpha}$
with $f(X_z) = 2$ 
contains exactly $\tau$ vertices $z$ 
such that $h(z) =1$ 
and $p- \tau $ vertices $z$ 
for which $h(z) = 2$.
The remaining set  
$\mathcal{Z}_i^{\alpha}$ 
with $z$ such that 
$f(X_z) =1$
contains only vertices
colored into 
the first color by $h$. 
Summing up,
each such vertex $y$ 
with $h(y) =1$ 
is adjacent to 
$(q^2-1) (p-\tau)$
vertices of the second
color in the coloring $h$.

 2. Let $y$ be a vertex 
 of $H(q+1,pq)$ 
 colored by $h$ 
 into the second color 
 and  let us count 
 the number of vertices $z$ 
 adjacent to $y$ 
 and colored 
 to the first color.
 Note that for all 
 such vertices $y$ 
 we have $f(X_y) = 2$.
 
 By the definition
 of the coloring $h$, 
 there are $q \tau$ vertices 
 $z$ adjacent to $y$ 
 such that  
 $y_i \equiv z_i \mod q$
 for all $i$
 and $h(z) =1$.
 Consider the sets 
 $\mathcal{Z}_i^{\alpha}$
 of vertices $z$ 
 adjacent to $y$ 
 such that 
 $z_i  \equiv \alpha \not\equiv y_i \mod q$. 
 If $i = i_{X_y}$
 is the special direction
 for the vertex $X_y$ 
 in $H(q+1,q)$,
 then for all 
 $\alpha \not\equiv y_i \mod q$
 all vertices $z$ 
 from the sets $\mathcal{Z}_i^{\alpha}$ 
 are colored into 
 the second color by $h$.
 On the other hand, 
 if $i \neq i_{X_y}$,
 then each of
 $q^2 - q -1$ sets $\mathcal{Z}_i^{\alpha}$ with $f(X_z) = 2$ 
 contains exactly $\tau$
 vertices $z$ 
 such that $h(z) =1$ and $p-\tau$ 
 vertices $z$ 
 for which $h(z) = 2$. 
 As before, 
 the remaining set  
 $\mathcal{Z}_i^{\alpha}$
 with $z$ satisfying
 $f(X_z) =1$ 
 contains only vertices 
 colored by $h$ 
 into the first color.  
 Summing up, 
 each vertex $y$ 
 with $h(y) =2$ 
 is adjacent to 
 $(q^2-1) \tau +p$ 
 vertices of the first color 
 in the coloring $h$.
 
 The second eigenvalue 
 of $h$ is
 $$
 \tilde{\theta} = (q+1) (pq-1) - q^2 p = q(p-1) - 1.
 $$
\end{proof}

\begin{remark}
The second color of the constructed coloring $h$ can be also partitioned into lines. So Theorem~\ref{splitconstr} allows an iterative application.
\end{remark}
 
Let us say few words on parameters for which the above constructions are applicable.

By Proposition~\ref{dcmpboolclr},
we can use (once or recursively)
the construction from 
Theorem~\ref{Flaassimprov}(1) 
for all $(b,c)$-colorings in $H(n,2)$ 
with the 
non-positive second eigenvalue. 
Theorem~\ref{Flaassimprov}(2)
can be used 
for constructing 
perfect colorings 
from colorings given after 
the application of 
Theorem~\ref{splitconstr},
because this theorem gives
colorings with large 
second eigenvalue 
and such that one 
of their colors 
can be split into faces.

As is stated in 
Proposition~\ref{dcmpboolclr}, 
the second color 
of a $(3,1)$-coloring of $H(3,2)$ 
can be partitioned 
into edges
($1$-faces, lines). 

The second color 
of a $(8,1)$-coloring 
(the complement 
of a $1$-perfect code) 
in $H(4,3)$ 
can be partitioned 
into lines. 
It is straightforward to verify 
that all such partitions
are equivalent 
to the following

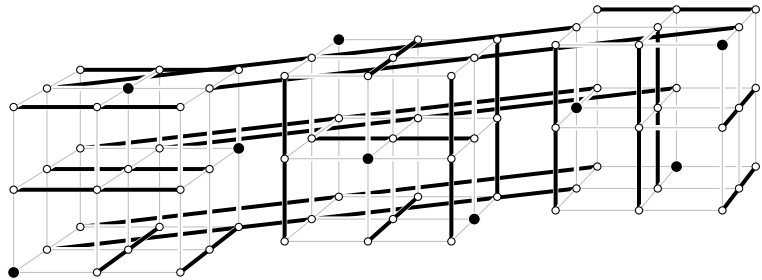
\def\dxd{13}
\def\dyd{1.5}
\def\dxa{4}
\def\dya{0}
\def\dxc{0}
\def\dyc{4}
\def\dxb{1.6}
\def\dyb{1.1}
\def\dxba{-0.1}
\def\dyba{0}
\def\dxbc{0}
\def\dybc{-0.1}
\def\dxbd{-0.3}
\def\dybd{-0.02}

$$
\begin{tikzpicture}[scale=0.6, x=1.3em,y=1.3em,
 cll/.style={circle,draw=white, thin, minimum size=0.1em, inner sep=0pt},
 thn/.style={draw=gray!50!white, thin},
 thk/.style={draw=black,  line width=1.5pt},
 wtl/.style={draw=white,  line width=2pt}]
\node [cll] (v0000) at (0*\dxa+0*\dxb+0*\dxc+0*\dxd+0*0*\dxba+0*0*\dxbc+0*0*\dxbd,0*\dya+0*\dyb+0*\dyc+0*\dyd+0*0*\dyba+0*0*\dybc+0*0*\dybd) {};
\node [cll] (v1000) at (1*\dxa+0*\dxb+0*\dxc+0*\dxd+0*1*\dxba+0*0*\dxbc+0*0*\dxbd,1*\dya+0*\dyb+0*\dyc+0*\dyd+0*1*\dyba+0*0*\dybc+0*0*\dybd) {};
\node [cll] (v2000) at (2*\dxa+0*\dxb+0*\dxc+0*\dxd+0*2*\dxba+0*0*\dxbc+0*0*\dxbd,2*\dya+0*\dyb+0*\dyc+0*\dyd+0*2*\dyba+0*0*\dybc+0*0*\dybd) {};
\node [cll] (v0010) at (0*\dxa+0*\dxb+1*\dxc+0*\dxd+0*0*\dxba+0*1*\dxbc+0*0*\dxbd,0*\dya+0*\dyb+1*\dyc+0*\dyd+0*0*\dyba+0*1*\dybc+0*0*\dybd) {};
\node [cll] (v1010) at (1*\dxa+0*\dxb+1*\dxc+0*\dxd+0*1*\dxba+0*1*\dxbc+0*0*\dxbd,1*\dya+0*\dyb+1*\dyc+0*\dyd+0*1*\dyba+0*1*\dybc+0*0*\dybd) {};
\node [cll] (v2010) at (2*\dxa+0*\dxb+1*\dxc+0*\dxd+0*2*\dxba+0*1*\dxbc+0*0*\dxbd,2*\dya+0*\dyb+1*\dyc+0*\dyd+0*2*\dyba+0*1*\dybc+0*0*\dybd) {};
\node [cll] (v0020) at (0*\dxa+0*\dxb+2*\dxc+0*\dxd+0*0*\dxba+0*2*\dxbc+0*0*\dxbd,0*\dya+0*\dyb+2*\dyc+0*\dyd+0*0*\dyba+0*2*\dybc+0*0*\dybd) {};
\node [cll] (v1020) at (1*\dxa+0*\dxb+2*\dxc+0*\dxd+0*1*\dxba+0*2*\dxbc+0*0*\dxbd,1*\dya+0*\dyb+2*\dyc+0*\dyd+0*1*\dyba+0*2*\dybc+0*0*\dybd) {};
\node [cll] (v2020) at (2*\dxa+0*\dxb+2*\dxc+0*\dxd+0*2*\dxba+0*2*\dxbc+0*0*\dxbd,2*\dya+0*\dyb+2*\dyc+0*\dyd+0*2*\dyba+0*2*\dybc+0*0*\dybd) {};
\node [cll] (v0100) at (0*\dxa+1*\dxb+0*\dxc+0*\dxd+1*0*\dxba+1*0*\dxbc+1*0*\dxbd,0*\dya+1*\dyb+0*\dyc+0*\dyd+1*0*\dyba+1*0*\dybc+1*0*\dybd) {};
\node [cll] (v1100) at (1*\dxa+1*\dxb+0*\dxc+0*\dxd+1*1*\dxba+1*0*\dxbc+1*0*\dxbd,1*\dya+1*\dyb+0*\dyc+0*\dyd+1*1*\dyba+1*0*\dybc+1*0*\dybd) {};
\node [cll] (v2100) at (2*\dxa+1*\dxb+0*\dxc+0*\dxd+1*2*\dxba+1*0*\dxbc+1*0*\dxbd,2*\dya+1*\dyb+0*\dyc+0*\dyd+1*2*\dyba+1*0*\dybc+1*0*\dybd) {};
\node [cll] (v0110) at (0*\dxa+1*\dxb+1*\dxc+0*\dxd+1*0*\dxba+1*1*\dxbc+1*0*\dxbd,0*\dya+1*\dyb+1*\dyc+0*\dyd+1*0*\dyba+1*1*\dybc+1*0*\dybd) {};
\node [cll] (v1110) at (1*\dxa+1*\dxb+1*\dxc+0*\dxd+1*1*\dxba+1*1*\dxbc+1*0*\dxbd,1*\dya+1*\dyb+1*\dyc+0*\dyd+1*1*\dyba+1*1*\dybc+1*0*\dybd) {};
\node [cll] (v2110) at (2*\dxa+1*\dxb+1*\dxc+0*\dxd+1*2*\dxba+1*1*\dxbc+1*0*\dxbd,2*\dya+1*\dyb+1*\dyc+0*\dyd+1*2*\dyba+1*1*\dybc+1*0*\dybd) {};
\node [cll] (v0120) at (0*\dxa+1*\dxb+2*\dxc+0*\dxd+1*0*\dxba+1*2*\dxbc+1*0*\dxbd,0*\dya+1*\dyb+2*\dyc+0*\dyd+1*0*\dyba+1*2*\dybc+1*0*\dybd) {};
\node [cll] (v1120) at (1*\dxa+1*\dxb+2*\dxc+0*\dxd+1*1*\dxba+1*2*\dxbc+1*0*\dxbd,1*\dya+1*\dyb+2*\dyc+0*\dyd+1*1*\dyba+1*2*\dybc+1*0*\dybd) {};
\node [cll] (v2120) at (2*\dxa+1*\dxb+2*\dxc+0*\dxd+1*2*\dxba+1*2*\dxbc+1*0*\dxbd,2*\dya+1*\dyb+2*\dyc+0*\dyd+1*2*\dyba+1*2*\dybc+1*0*\dybd) {};
\node [cll] (v0200) at (0*\dxa+2*\dxb+0*\dxc+0*\dxd+2*0*\dxba+2*0*\dxbc+2*0*\dxbd,0*\dya+2*\dyb+0*\dyc+0*\dyd+2*0*\dyba+2*0*\dybc+2*0*\dybd) {};
\node [cll] (v1200) at (1*\dxa+2*\dxb+0*\dxc+0*\dxd+2*1*\dxba+2*0*\dxbc+2*0*\dxbd,1*\dya+2*\dyb+0*\dyc+0*\dyd+2*1*\dyba+2*0*\dybc+2*0*\dybd) {};
\node [cll] (v2200) at (2*\dxa+2*\dxb+0*\dxc+0*\dxd+2*2*\dxba+2*0*\dxbc+2*0*\dxbd,2*\dya+2*\dyb+0*\dyc+0*\dyd+2*2*\dyba+2*0*\dybc+2*0*\dybd) {};
\node [cll] (v0210) at (0*\dxa+2*\dxb+1*\dxc+0*\dxd+2*0*\dxba+2*1*\dxbc+2*0*\dxbd,0*\dya+2*\dyb+1*\dyc+0*\dyd+2*0*\dyba+2*1*\dybc+2*0*\dybd) {};
\node [cll] (v1210) at (1*\dxa+2*\dxb+1*\dxc+0*\dxd+2*1*\dxba+2*1*\dxbc+2*0*\dxbd,1*\dya+2*\dyb+1*\dyc+0*\dyd+2*1*\dyba+2*1*\dybc+2*0*\dybd) {};
\node [cll] (v2210) at (2*\dxa+2*\dxb+1*\dxc+0*\dxd+2*2*\dxba+2*1*\dxbc+2*0*\dxbd,2*\dya+2*\dyb+1*\dyc+0*\dyd+2*2*\dyba+2*1*\dybc+2*0*\dybd) {};
\node [cll] (v0220) at (0*\dxa+2*\dxb+2*\dxc+0*\dxd+2*0*\dxba+2*2*\dxbc+2*0*\dxbd,0*\dya+2*\dyb+2*\dyc+0*\dyd+2*0*\dyba+2*2*\dybc+2*0*\dybd) {};
\node [cll] (v1220) at (1*\dxa+2*\dxb+2*\dxc+0*\dxd+2*1*\dxba+2*2*\dxbc+2*0*\dxbd,1*\dya+2*\dyb+2*\dyc+0*\dyd+2*1*\dyba+2*2*\dybc+2*0*\dybd) {};
\node [cll] (v2220) at (2*\dxa+2*\dxb+2*\dxc+0*\dxd+2*2*\dxba+2*2*\dxbc+2*0*\dxbd,2*\dya+2*\dyb+2*\dyc+0*\dyd+2*2*\dyba+2*2*\dybc+2*0*\dybd) {};

\node [cll] (v0001) at (0*\dxa+0*\dxb+0*\dxc+1*\dxd+0*0*\dxba+0*0*\dxbc+0*1*\dxbd,0*\dya+0*\dyb+0*\dyc+1*\dyd+0*0*\dyba+0*0*\dybc+0*1*\dybd) {};
\node [cll] (v1001) at (1*\dxa+0*\dxb+0*\dxc+1*\dxd+0*1*\dxba+0*0*\dxbc+0*1*\dxbd,1*\dya+0*\dyb+0*\dyc+1*\dyd+0*1*\dyba+0*0*\dybc+0*1*\dybd) {};
\node [cll] (v2001) at (2*\dxa+0*\dxb+0*\dxc+1*\dxd+0*2*\dxba+0*0*\dxbc+0*1*\dxbd,2*\dya+0*\dyb+0*\dyc+1*\dyd+0*2*\dyba+0*0*\dybc+0*1*\dybd) {};
\node [cll] (v0011) at (0*\dxa+0*\dxb+1*\dxc+1*\dxd+0*0*\dxba+0*1*\dxbc+0*1*\dxbd,0*\dya+0*\dyb+1*\dyc+1*\dyd+0*0*\dyba+0*1*\dybc+0*1*\dybd) {};
\node [cll] (v1011) at (1*\dxa+0*\dxb+1*\dxc+1*\dxd+0*1*\dxba+0*1*\dxbc+0*1*\dxbd,1*\dya+0*\dyb+1*\dyc+1*\dyd+0*1*\dyba+0*1*\dybc+0*1*\dybd) {};
\node [cll] (v2011) at (2*\dxa+0*\dxb+1*\dxc+1*\dxd+0*2*\dxba+0*1*\dxbc+0*1*\dxbd,2*\dya+0*\dyb+1*\dyc+1*\dyd+0*2*\dyba+0*1*\dybc+0*1*\dybd) {};
\node [cll] (v0021) at (0*\dxa+0*\dxb+2*\dxc+1*\dxd+0*0*\dxba+0*2*\dxbc+0*1*\dxbd,0*\dya+0*\dyb+2*\dyc+1*\dyd+0*0*\dyba+0*2*\dybc+0*1*\dybd) {};
\node [cll] (v1021) at (1*\dxa+0*\dxb+2*\dxc+1*\dxd+0*1*\dxba+0*2*\dxbc+0*1*\dxbd,1*\dya+0*\dyb+2*\dyc+1*\dyd+0*1*\dyba+0*2*\dybc+0*1*\dybd) {};
\node [cll] (v2021) at (2*\dxa+0*\dxb+2*\dxc+1*\dxd+0*2*\dxba+0*2*\dxbc+0*1*\dxbd,2*\dya+0*\dyb+2*\dyc+1*\dyd+0*2*\dyba+0*2*\dybc+0*1*\dybd) {};
\node [cll] (v0101) at (0*\dxa+1*\dxb+0*\dxc+1*\dxd+1*0*\dxba+1*0*\dxbc+1*1*\dxbd,0*\dya+1*\dyb+0*\dyc+1*\dyd+1*0*\dyba+1*0*\dybc+1*1*\dybd) {};
\node [cll] (v1101) at (1*\dxa+1*\dxb+0*\dxc+1*\dxd+1*1*\dxba+1*0*\dxbc+1*1*\dxbd,1*\dya+1*\dyb+0*\dyc+1*\dyd+1*1*\dyba+1*0*\dybc+1*1*\dybd) {};
\node [cll] (v2101) at (2*\dxa+1*\dxb+0*\dxc+1*\dxd+1*2*\dxba+1*0*\dxbc+1*1*\dxbd,2*\dya+1*\dyb+0*\dyc+1*\dyd+1*2*\dyba+1*0*\dybc+1*1*\dybd) {};
\node [cll] (v0111) at (0*\dxa+1*\dxb+1*\dxc+1*\dxd+1*0*\dxba+1*1*\dxbc+1*1*\dxbd,0*\dya+1*\dyb+1*\dyc+1*\dyd+1*0*\dyba+1*1*\dybc+1*1*\dybd) {};
\node [cll] (v1111) at (1*\dxa+1*\dxb+1*\dxc+1*\dxd+1*1*\dxba+1*1*\dxbc+1*1*\dxbd,1*\dya+1*\dyb+1*\dyc+1*\dyd+1*1*\dyba+1*1*\dybc+1*1*\dybd) {};
\node [cll] (v2111) at (2*\dxa+1*\dxb+1*\dxc+1*\dxd+1*2*\dxba+1*1*\dxbc+1*1*\dxbd,2*\dya+1*\dyb+1*\dyc+1*\dyd+1*2*\dyba+1*1*\dybc+1*1*\dybd) {};
\node [cll] (v0121) at (0*\dxa+1*\dxb+2*\dxc+1*\dxd+1*0*\dxba+1*2*\dxbc+1*1*\dxbd,0*\dya+1*\dyb+2*\dyc+1*\dyd+1*0*\dyba+1*2*\dybc+1*1*\dybd) {};
\node [cll] (v1121) at (1*\dxa+1*\dxb+2*\dxc+1*\dxd+1*1*\dxba+1*2*\dxbc+1*1*\dxbd,1*\dya+1*\dyb+2*\dyc+1*\dyd+1*1*\dyba+1*2*\dybc+1*1*\dybd) {};
\node [cll] (v2121) at (2*\dxa+1*\dxb+2*\dxc+1*\dxd+1*2*\dxba+1*2*\dxbc+1*1*\dxbd,2*\dya+1*\dyb+2*\dyc+1*\dyd+1*2*\dyba+1*2*\dybc+1*1*\dybd) {};
\node [cll] (v0201) at (0*\dxa+2*\dxb+0*\dxc+1*\dxd+2*0*\dxba+2*0*\dxbc+2*1*\dxbd,0*\dya+2*\dyb+0*\dyc+1*\dyd+2*0*\dyba+2*0*\dybc+2*1*\dybd) {};
\node [cll] (v1201) at (1*\dxa+2*\dxb+0*\dxc+1*\dxd+2*1*\dxba+2*0*\dxbc+2*1*\dxbd,1*\dya+2*\dyb+0*\dyc+1*\dyd+2*1*\dyba+2*0*\dybc+2*1*\dybd) {};
\node [cll] (v2201) at (2*\dxa+2*\dxb+0*\dxc+1*\dxd+2*2*\dxba+2*0*\dxbc+2*1*\dxbd,2*\dya+2*\dyb+0*\dyc+1*\dyd+2*2*\dyba+2*0*\dybc+2*1*\dybd) {};
\node [cll] (v0211) at (0*\dxa+2*\dxb+1*\dxc+1*\dxd+2*0*\dxba+2*1*\dxbc+2*1*\dxbd,0*\dya+2*\dyb+1*\dyc+1*\dyd+2*0*\dyba+2*1*\dybc+2*1*\dybd) {};
\node [cll] (v1211) at (1*\dxa+2*\dxb+1*\dxc+1*\dxd+2*1*\dxba+2*1*\dxbc+2*1*\dxbd,1*\dya+2*\dyb+1*\dyc+1*\dyd+2*1*\dyba+2*1*\dybc+2*1*\dybd) {};
\node [cll] (v2211) at (2*\dxa+2*\dxb+1*\dxc+1*\dxd+2*2*\dxba+2*1*\dxbc+2*1*\dxbd,2*\dya+2*\dyb+1*\dyc+1*\dyd+2*2*\dyba+2*1*\dybc+2*1*\dybd) {};
\node [cll] (v0221) at (0*\dxa+2*\dxb+2*\dxc+1*\dxd+2*0*\dxba+2*2*\dxbc+2*1*\dxbd,0*\dya+2*\dyb+2*\dyc+1*\dyd+2*0*\dyba+2*2*\dybc+2*1*\dybd) {};
\node [cll] (v1221) at (1*\dxa+2*\dxb+2*\dxc+1*\dxd+2*1*\dxba+2*2*\dxbc+2*1*\dxbd,1*\dya+2*\dyb+2*\dyc+1*\dyd+2*1*\dyba+2*2*\dybc+2*1*\dybd) {};
\node [cll] (v2221) at (2*\dxa+2*\dxb+2*\dxc+1*\dxd+2*2*\dxba+2*2*\dxbc+2*1*\dxbd,2*\dya+2*\dyb+2*\dyc+1*\dyd+2*2*\dyba+2*2*\dybc+2*1*\dybd) {};

\node [cll] (v0002) at (0*\dxa+0*\dxb+0*\dxc+2*\dxd+0*0*\dxba+0*0*\dxbc+0*2*\dxbd,0*\dya+0*\dyb+0*\dyc+2*\dyd+0*0*\dyba+0*0*\dybc+0*2*\dybd) {};
\node [cll] (v1002) at (1*\dxa+0*\dxb+0*\dxc+2*\dxd+0*1*\dxba+0*0*\dxbc+0*2*\dxbd,1*\dya+0*\dyb+0*\dyc+2*\dyd+0*1*\dyba+0*0*\dybc+0*2*\dybd) {};
\node [cll] (v2002) at (2*\dxa+0*\dxb+0*\dxc+2*\dxd+0*2*\dxba+0*0*\dxbc+0*2*\dxbd,2*\dya+0*\dyb+0*\dyc+2*\dyd+0*2*\dyba+0*0*\dybc+0*2*\dybd) {};
\node [cll] (v0012) at (0*\dxa+0*\dxb+1*\dxc+2*\dxd+0*0*\dxba+0*1*\dxbc+0*2*\dxbd,0*\dya+0*\dyb+1*\dyc+2*\dyd+0*0*\dyba+0*1*\dybc+0*2*\dybd) {};
\node [cll] (v1012) at (1*\dxa+0*\dxb+1*\dxc+2*\dxd+0*1*\dxba+0*1*\dxbc+0*2*\dxbd,1*\dya+0*\dyb+1*\dyc+2*\dyd+0*1*\dyba+0*1*\dybc+0*2*\dybd) {};
\node [cll] (v2012) at (2*\dxa+0*\dxb+1*\dxc+2*\dxd+0*2*\dxba+0*1*\dxbc+0*2*\dxbd,2*\dya+0*\dyb+1*\dyc+2*\dyd+0*2*\dyba+0*1*\dybc+0*2*\dybd) {};
\node [cll] (v0022) at (0*\dxa+0*\dxb+2*\dxc+2*\dxd+0*0*\dxba+0*2*\dxbc+0*2*\dxbd,0*\dya+0*\dyb+2*\dyc+2*\dyd+0*0*\dyba+0*2*\dybc+0*2*\dybd) {};
\node [cll] (v1022) at (1*\dxa+0*\dxb+2*\dxc+2*\dxd+0*1*\dxba+0*2*\dxbc+0*2*\dxbd,1*\dya+0*\dyb+2*\dyc+2*\dyd+0*1*\dyba+0*2*\dybc+0*2*\dybd) {};
\node [cll] (v2022) at (2*\dxa+0*\dxb+2*\dxc+2*\dxd+0*2*\dxba+0*2*\dxbc+0*2*\dxbd,2*\dya+0*\dyb+2*\dyc+2*\dyd+0*2*\dyba+0*2*\dybc+0*2*\dybd) {};
\node [cll] (v0102) at (0*\dxa+1*\dxb+0*\dxc+2*\dxd+1*0*\dxba+1*0*\dxbc+1*2*\dxbd,0*\dya+1*\dyb+0*\dyc+2*\dyd+1*0*\dyba+1*0*\dybc+1*2*\dybd) {};
\node [cll] (v1102) at (1*\dxa+1*\dxb+0*\dxc+2*\dxd+1*1*\dxba+1*0*\dxbc+1*2*\dxbd,1*\dya+1*\dyb+0*\dyc+2*\dyd+1*1*\dyba+1*0*\dybc+1*2*\dybd) {};
\node [cll] (v2102) at (2*\dxa+1*\dxb+0*\dxc+2*\dxd+1*2*\dxba+1*0*\dxbc+1*2*\dxbd,2*\dya+1*\dyb+0*\dyc+2*\dyd+1*2*\dyba+1*0*\dybc+1*2*\dybd) {};
\node [cll] (v0112) at (0*\dxa+1*\dxb+1*\dxc+2*\dxd+1*0*\dxba+1*1*\dxbc+1*2*\dxbd,0*\dya+1*\dyb+1*\dyc+2*\dyd+1*0*\dyba+1*1*\dybc+1*2*\dybd) {};
\node [cll] (v1112) at (1*\dxa+1*\dxb+1*\dxc+2*\dxd+1*1*\dxba+1*1*\dxbc+1*2*\dxbd,1*\dya+1*\dyb+1*\dyc+2*\dyd+1*1*\dyba+1*1*\dybc+1*2*\dybd) {};
\node [cll] (v2112) at (2*\dxa+1*\dxb+1*\dxc+2*\dxd+1*2*\dxba+1*1*\dxbc+1*2*\dxbd,2*\dya+1*\dyb+1*\dyc+2*\dyd+1*2*\dyba+1*1*\dybc+1*2*\dybd) {};
\node [cll] (v0122) at (0*\dxa+1*\dxb+2*\dxc+2*\dxd+1*0*\dxba+1*2*\dxbc+1*2*\dxbd,0*\dya+1*\dyb+2*\dyc+2*\dyd+1*0*\dyba+1*2*\dybc+1*2*\dybd) {};
\node [cll] (v1122) at (1*\dxa+1*\dxb+2*\dxc+2*\dxd+1*1*\dxba+1*2*\dxbc+1*2*\dxbd,1*\dya+1*\dyb+2*\dyc+2*\dyd+1*1*\dyba+1*2*\dybc+1*2*\dybd) {};
\node [cll] (v2122) at (2*\dxa+1*\dxb+2*\dxc+2*\dxd+1*2*\dxba+1*2*\dxbc+1*2*\dxbd,2*\dya+1*\dyb+2*\dyc+2*\dyd+1*2*\dyba+1*2*\dybc+1*2*\dybd) {};
\node [cll] (v0202) at (0*\dxa+2*\dxb+0*\dxc+2*\dxd+2*0*\dxba+2*0*\dxbc+2*2*\dxbd,0*\dya+2*\dyb+0*\dyc+2*\dyd+2*0*\dyba+2*0*\dybc+2*2*\dybd) {};
\node [cll] (v1202) at (1*\dxa+2*\dxb+0*\dxc+2*\dxd+2*1*\dxba+2*0*\dxbc+2*2*\dxbd,1*\dya+2*\dyb+0*\dyc+2*\dyd+2*1*\dyba+2*0*\dybc+2*2*\dybd) {};
\node [cll] (v2202) at (2*\dxa+2*\dxb+0*\dxc+2*\dxd+2*2*\dxba+2*0*\dxbc+2*2*\dxbd,2*\dya+2*\dyb+0*\dyc+2*\dyd+2*2*\dyba+2*0*\dybc+2*2*\dybd) {};
\node [cll] (v0212) at (0*\dxa+2*\dxb+1*\dxc+2*\dxd+2*0*\dxba+2*1*\dxbc+2*2*\dxbd,0*\dya+2*\dyb+1*\dyc+2*\dyd+2*0*\dyba+2*1*\dybc+2*2*\dybd) {};
\node [cll] (v1212) at (1*\dxa+2*\dxb+1*\dxc+2*\dxd+2*1*\dxba+2*1*\dxbc+2*2*\dxbd,1*\dya+2*\dyb+1*\dyc+2*\dyd+2*1*\dyba+2*1*\dybc+2*2*\dybd) {};
\node [cll] (v2212) at (2*\dxa+2*\dxb+1*\dxc+2*\dxd+2*2*\dxba+2*1*\dxbc+2*2*\dxbd,2*\dya+2*\dyb+1*\dyc+2*\dyd+2*2*\dyba+2*1*\dybc+2*2*\dybd) {};
\node [cll] (v0222) at (0*\dxa+2*\dxb+2*\dxc+2*\dxd+2*0*\dxba+2*2*\dxbc+2*2*\dxbd,0*\dya+2*\dyb+2*\dyc+2*\dyd+2*0*\dyba+2*2*\dybc+2*2*\dybd) {};
\node [cll] (v1222) at (1*\dxa+2*\dxb+2*\dxc+2*\dxd+2*1*\dxba+2*2*\dxbc+2*2*\dxbd,1*\dya+2*\dyb+2*\dyc+2*\dyd+2*1*\dyba+2*2*\dybc+2*2*\dybd) {};
\node [cll] (v2222) at (2*\dxa+2*\dxb+2*\dxc+2*\dxd+2*2*\dxba+2*2*\dxbc+2*2*\dxbd,2*\dya+2*\dyb+2*\dyc+2*\dyd+2*2*\dyba+2*2*\dybc+2*2*\dybd) {};

 \draw [wtl] (v0202)--(v0212)--(v0222); \draw [thn] (v0202)--(v0212)--(v0222);
 \draw [wtl] (v1202)--(v1212)--(v1222); \draw [thn] (v1202)--(v1212)--(v1222);
 \draw [wtl] (v2202)--(v2212)--(v2222); \draw [thn] (v2202)--(v2212)--(v2222);
\draw [wtl] (v0202)--(v1202)--(v2202); \draw [thn] (v0202)--(v1202)--(v2202);
\draw [wtl] (v0212)--(v1212)--(v2212); \draw [thn] (v0212)--(v1212)--(v2212);
\draw [wtl] (v0222)--(v1222)--(v2222); \draw [thk] (v0222)--(v1222)--(v2222);

\draw [thk] (v0201)--(v0202);
\draw [thk] (v0211)--(v0212);
\draw [thk] (v1211)--(v1212);

 \draw [wtl] (v0201)--(v0211)--(v0221); \draw [thn] (v0201)--(v0211)--(v0221);
 \draw [wtl] (v1201)--(v1211)--(v1221); \draw [thn] (v1201)--(v1211)--(v1221);
 \draw [wtl] (v2201)--(v2211)--(v2221); \draw [thk] (v2201)--(v2211)--(v2221);
\draw [wtl] (v0201)--(v1201)--(v2201); \draw [thn] (v0201)--(v1201)--(v2201);
\draw [wtl] (v0211)--(v1211)--(v2211); \draw [thn] (v0211)--(v1211)--(v2211);
\draw [wtl] (v0221)--(v1221)--(v2221); \draw [thn] (v0221)--(v1221)--(v2221);

\draw [thk] (v0200)--(v0201);
\draw [thk] (v0210)--(v0211);
\draw [thk] (v1210)--(v1211);

 \draw [wtl] (v0200)--(v0210)--(v0220); \draw [thn] (v0200)--(v0210)--(v0220);
 \draw [wtl] (v1200)--(v1210)--(v1220); \draw [thn] (v1200)--(v1210)--(v1220);
 \draw [wtl] (v2200)--(v2210)--(v2220); \draw [thn] (v2200)--(v2210)--(v2220);
\draw [wtl] (v0200)--(v1200)--(v2200); \draw [thn] (v0200)--(v1200)--(v2200);
\draw [wtl] (v0210)--(v1210)--(v2210); \draw [thn] (v0210)--(v1210)--(v2210);
\draw [wtl] (v0220)--(v1220)--(v2220); \draw [thk] (v0220)--(v1220)--(v2220);

 \draw [wtl] (v0102)--(v0112)--(v0122); \draw [thn] (v0102)--(v0112)--(v0122);
 \draw [wtl] (v1102)--(v1112)--(v1122); \draw [thk] (v1102)--(v1112)--(v1122);
 \draw [wtl] (v2102)--(v2112)--(v2122); \draw [thn] (v2102)--(v2112)--(v2122);
\draw [wtl] (v0102)--(v1102)--(v2102); \draw [thn] (v0102)--(v1102)--(v2102);
\draw [wtl] (v0112)--(v1112)--(v2112); \draw [thn] (v0112)--(v1112)--(v2112);
\draw [wtl] (v0122)--(v1122)--(v2122); \draw [thn] (v0122)--(v1122)--(v2122);

\draw [thk] (v0121)--(v0122);
\draw [thk] (v2121)--(v2122);
\draw [thk] (v0101)--(v0102);

 \draw [wtl] (v0101)--(v0111)--(v0121); \draw [thn] (v0101)--(v0111)--(v0121);
 \draw [wtl] (v1101)--(v1111)--(v1121); \draw [thn] (v1101)--(v1111)--(v1121);
 \draw [wtl] (v2101)--(v2111)--(v2121); \draw [thn] (v2101)--(v2111)--(v2121);
\draw [wtl] (v0101)--(v1101)--(v2101); \draw [thn] (v0101)--(v1101)--(v2101);
\draw [wtl] (v0111)--(v1111)--(v2111); \draw [thk] (v0111)--(v1111)--(v2111);
\draw [wtl] (v0121)--(v1121)--(v2121); \draw [thn] (v0121)--(v1121)--(v2121);

\draw [thk] (v0120)--(v0121);
\draw [thk] (v2120)--(v2121);
\draw [thk] (v0100)--(v0101);

 \draw [wtl] (v0100)--(v0110)--(v0120); \draw [thn] (v0100)--(v0110)--(v0120);
 \draw [wtl] (v1100)--(v1110)--(v1120); \draw [thn] (v1100)--(v1110)--(v1120);
 \draw [wtl] (v2100)--(v2110)--(v2120); \draw [thn] (v2100)--(v2110)--(v2120);
\draw [wtl] (v0100)--(v1100)--(v2100); \draw [thn] (v0100)--(v1100)--(v2100);
\draw [wtl] (v0110)--(v1110)--(v2110); \draw [thk] (v0110)--(v1110)--(v2110);
\draw [wtl] (v0120)--(v1120)--(v2120); \draw [thn] (v0120)--(v1120)--(v2120);

 \draw [wtl] (v0000)--(v0010)--(v0020); \draw [thn] (v0000)--(v0010)--(v0020);
 \draw [wtl] (v1000)--(v1010)--(v1020); \draw [thn] (v1000)--(v1010)--(v1020);
 \draw [wtl] (v2000)--(v2010)--(v2020); \draw [thn] (v2000)--(v2010)--(v2020);
\draw [wtl] (v0000)--(v1000)--(v2000); \draw [thn] (v0000)--(v1000)--(v2000);
\draw [wtl] (v0010)--(v1010)--(v2010); \draw [thk] (v0010)--(v1010)--(v2010);
\draw [wtl] (v0020)--(v1020)--(v2020); \draw [thk] (v0020)--(v1020)--(v2020);

 \draw [wtl] (v0001)--(v0011)--(v0021); \draw [thk] (v0001)--(v0011)--(v0021);
 \draw [wtl] (v1001)--(v1011)--(v1021); \draw [thn] (v1001)--(v1011)--(v1021);
 \draw [wtl] (v2001)--(v2011)--(v2021); \draw [thk] (v2001)--(v2011)--(v2021);
\draw [wtl] (v0001)--(v1001)--(v2001); \draw [thn] (v0001)--(v1001)--(v2001);
\draw [wtl] (v0011)--(v1011)--(v2011); \draw [thn] (v0011)--(v1011)--(v2011);
\draw [wtl] (v0021)--(v1021)--(v2021); \draw [thn] (v0021)--(v1021)--(v2021);

 \draw [wtl] (v0002)--(v0012)--(v0022); \draw [thk] (v0002)--(v0012)--(v0022);
 \draw [wtl] (v1002)--(v1012)--(v1022); \draw [thk] (v1002)--(v1012)--(v1022);
 \draw [wtl] (v2002)--(v2012)--(v2022); \draw [thn] (v2002)--(v2012)--(v2022);
\draw [wtl] (v0002)--(v1002)--(v2002); \draw [thn] (v0002)--(v1002)--(v2002);
\draw [wtl] (v0012)--(v1012)--(v2012); \draw [thn] (v0012)--(v1012)--(v2012);
\draw [wtl] (v0022)--(v1022)--(v2022); \draw [thn] (v0022)--(v1022)--(v2022);

\draw [wtl] (v0000)--(v0100)--(v0200); \draw [thn] (v0000)--(v0100)--(v0200);
\draw [wtl] (v1000)--(v1100)--(v1200); \draw [thk] (v1000)--(v1100)--(v1200);
\draw [wtl] (v2000)--(v2100)--(v2200); \draw [thk] (v2000)--(v2100)--(v2200);
\draw [wtl] (v0010)--(v0110)--(v0210); \draw [thn] (v0010)--(v0110)--(v0210);
\draw [wtl] (v1010)--(v1110)--(v1210); \draw [thn] (v1010)--(v1110)--(v1210);
\draw [wtl] (v2010)--(v2110)--(v2210); \draw [thn] (v2010)--(v2110)--(v2210);
\draw [wtl] (v0020)--(v0120)--(v0220); \draw [thn] (v0020)--(v0120)--(v0220);
\draw [wtl] (v1020)--(v1120)--(v1220); \draw [thn] (v1020)--(v1120)--(v1220);
\draw [wtl] (v2020)--(v2120)--(v2220); \draw [thn] (v2020)--(v2120)--(v2220);

\draw [wtl] (v0001)--(v0101)--(v0201); \draw [thn] (v0001)--(v0101)--(v0201);
\draw [wtl] (v1001)--(v1101)--(v1201); \draw [thk] (v1001)--(v1101)--(v1201);
\draw [wtl] (v2001)--(v2101)--(v2201); \draw [thn] (v2001)--(v2101)--(v2201);
\draw [wtl] (v0011)--(v0111)--(v0211); \draw [thn] (v0011)--(v0111)--(v0211);
\draw [wtl] (v1011)--(v1111)--(v1211); \draw [thn] (v1011)--(v1111)--(v1211);
\draw [wtl] (v2011)--(v2111)--(v2211); \draw [thn] (v2011)--(v2111)--(v2211);
\draw [wtl] (v0021)--(v0121)--(v0221); \draw [thn] (v0021)--(v0121)--(v0221);
\draw [wtl] (v1021)--(v1121)--(v1221); \draw [thk] (v1021)--(v1121)--(v1221);
\draw [wtl] (v2021)--(v2121)--(v2221); \draw [thn] (v2021)--(v2121)--(v2221);

\draw [wtl] (v0002)--(v0102)--(v0202); \draw [thn] (v0002)--(v0102)--(v0202);
\draw [wtl] (v1002)--(v1102)--(v1202); \draw [thn] (v1002)--(v1102)--(v1202);
\draw [wtl] (v2002)--(v2102)--(v2202); \draw [thk] (v2002)--(v2102)--(v2202);
\draw [wtl] (v0012)--(v0112)--(v0212); \draw [thn] (v0012)--(v0112)--(v0212);
\draw [wtl] (v1012)--(v1112)--(v1212); \draw [thn] (v1012)--(v1112)--(v1212);
\draw [wtl] (v2012)--(v2112)--(v2212); \draw [thk] (v2012)--(v2112)--(v2212);
\draw [wtl] (v0022)--(v0122)--(v0222); \draw [thn] (v0022)--(v0122)--(v0222);
\draw [wtl] (v1022)--(v1122)--(v1222); \draw [thn] (v1022)--(v1122)--(v1222);
\draw [wtl] (v2022)--(v2122)--(v2222); \draw [thn] (v2022)--(v2122)--(v2222);

\draw [fill=black,draw=white] (v0000) circle(1.3mm);
\draw [fill=white] (v1000) circle(0.8mm);
\draw [fill=white] (v2000) circle(0.8mm);
\draw [fill=white] (v0010) circle(0.8mm);
\draw [fill=white] (v1010) circle(0.8mm);
\draw [fill=white] (v2010) circle(0.8mm);
\draw [fill=white] (v0020) circle(0.8mm);
\draw [fill=white] (v1020) circle(0.8mm);
\draw [fill=white] (v2020) circle(0.8mm);
\draw [fill=white] (v0100) circle(0.8mm);
\draw [fill=white] (v1100) circle(0.8mm);
\draw [fill=white] (v2100) circle(0.8mm);
\draw [fill=white] (v0110) circle(0.8mm);
\draw [fill=white] (v1110) circle(0.8mm);
\draw [fill=white] (v2110) circle(0.8mm);
\draw [fill=white] (v0120) circle(0.8mm);
\draw [fill=black,draw=white] (v1120) circle(1.3mm);
\draw [fill=white] (v2120) circle(0.8mm);
\draw [fill=white] (v0200) circle(0.8mm);
\draw [fill=white] (v1200) circle(0.8mm);
\draw [fill=white] (v2200) circle(0.8mm);
\draw [fill=white] (v0210) circle(0.8mm);
\draw [fill=white] (v1210) circle(0.8mm);
\draw [fill=black,draw=white] (v2210) circle(1.3mm);
\draw [fill=white] (v0220) circle(0.8mm);
\draw [fill=white] (v1220) circle(0.8mm);
\draw [fill=white] (v2220) circle(0.8mm);

\draw [fill=white] (v0001) circle(0.8mm);
\draw [fill=white] (v1001) circle(0.8mm);
\draw [fill=white] (v2001) circle(0.8mm);
\draw [fill=white] (v0011) circle(0.8mm);
\draw [fill=black,draw=white] (v1011) circle(1.3mm);
\draw [fill=white] (v2011) circle(0.8mm);
\draw [fill=white] (v0021) circle(0.8mm);
\draw [fill=white] (v1021) circle(0.8mm);
\draw [fill=white] (v2021) circle(0.8mm);
\draw [fill=white] (v0101) circle(0.8mm);
\draw [fill=white] (v1101) circle(0.8mm);
\draw [fill=black,draw=white] (v2101) circle(1.3mm);
\draw [fill=white] (v0111) circle(0.8mm);
\draw [fill=white] (v1111) circle(0.8mm);
\draw [fill=white] (v2111) circle(0.8mm);
\draw [fill=white] (v0121) circle(0.8mm);
\draw [fill=white] (v1121) circle(0.8mm);
\draw [fill=white] (v2121) circle(0.8mm);
\draw [fill=white] (v0201) circle(0.8mm);
\draw [fill=white] (v1201) circle(0.8mm);
\draw [fill=white] (v2201) circle(0.8mm);
\draw [fill=white] (v0211) circle(0.8mm);
\draw [fill=white] (v1211) circle(0.8mm);
\draw [fill=white] (v2211) circle(0.8mm);
\draw [fill=black,draw=white] (v0221) circle(1.3mm);
\draw [fill=white] (v1221) circle(0.8mm);
\draw [fill=white] (v2221) circle(0.8mm);

\draw [fill=white] (v0002) circle(0.8mm);
\draw [fill=white] (v1002) circle(0.8mm);
\draw [fill=white] (v2002) circle(0.8mm);
\draw [fill=white] (v0012) circle(0.8mm);
\draw [fill=white] (v1012) circle(0.8mm);
\draw [fill=white] (v2012) circle(0.8mm);
\draw [fill=white] (v0022) circle(0.8mm);
\draw [fill=white] (v1022) circle(0.8mm);
\draw [fill=black,draw=white] (v2022) circle(1.3mm);
\draw [fill=white] (v0102) circle(0.8mm);
\draw [fill=white] (v1102) circle(0.8mm);
\draw [fill=white] (v2102) circle(0.8mm);
\draw [fill=black,draw=white] (v0112) circle(1.3mm);
\draw [fill=white] (v1112) circle(0.8mm);
\draw [fill=white] (v2112) circle(0.8mm);
\draw [fill=white] (v0122) circle(0.8mm);
\draw [fill=white] (v1122) circle(0.8mm);
\draw [fill=white] (v2122) circle(0.8mm);
\draw [fill=white] (v0202) circle(0.8mm);
\draw [fill=black,draw=white] (v1202) circle(1.3mm);
\draw [fill=white] (v2202) circle(0.8mm);
\draw [fill=white] (v0212) circle(0.8mm);
\draw [fill=white] (v1212) circle(0.8mm);
\draw [fill=white] (v2212) circle(0.8mm);
\draw [fill=white] (v0222) circle(0.8mm);
\draw [fill=white] (v1222) circle(0.8mm);
\draw [fill=white] (v2222) circle(0.8mm);
\end{tikzpicture}
$$

Therefore, Theorem~\ref{splitconstr} is applicable for $q=2$ or $q=3$ and gives us the following perfect colorings.
 
 \begin{corollary}
\item \begin{enumerate}
\item  For each $p \in \mathbb{N}$ and $\tau = 0, \ldots, p-1$ there exists a perfect $(3p-3 \tau,p+ 3 \tau)$-coloring of $H(3,2p)$ (with the second eigenvalue $\theta = 2p  - 3$).
\item   For each $p \in \mathbb{N}$ and $ \tau = 0, \ldots, p-1$ there exists a perfect $(8p-8 \tau , p + 8 \tau )$-coloring of $H(4,3p)$ 
(with the second eigenvalue $\theta = 3p - 4 $).
 \end{enumerate}
 \end{corollary}

For cases $q \geq 4$ we need to answer the following question.

\begin{question}
Given $q \geq 4$, does there exist a decomposition into $1$-dimensional faces of the complement of a $1$-perfect code in $H(q+1,q)$?
\end{question}

\subsection{All constructions}

In the following table we summarize all general constructions considered in the present paper. 
\begin{center}
\begin{tabular}{|p{1.39cm}|p{1.25cm}|p{1.07cm}|p{4.48cm}|p{2.3cm}|p{2.2cm}|p{2cm}|}\cline{1-7}
input $(b, c)$ & input $(n, q)$ & Th. & necessary conditions & output ($b$, $c$) & output ($n$, $q$) & output $\theta$\\ \hhline{|=|=|=|=|=|=|=|}
 --- &  --- &  $\tau$-fold perfect code &  $q=p^s$ is a prime power; $n=(q^m-1)/(q-1)$& $(n(q-1)- \tau +1,\tau)$ & $(n,q)$ & $-1$\\\cline{1-7}
 --- &  --- &  $\tau$-fold MDS code &  none & $(n(q-\tau), \tau n)$ & $(n,q)$ & $-n$\\\cline{1-7}
 --- &  --- & \ref{th:additive}  &  $q=p^s$ is a prime power; $\exists$ $c$-fold spread from $s$-dimensional subspaces of $\mathbb Z^m_p$
 & $(c(p^m-1), c)$ & $\displaystyle{\left(\frac{c(p^m-1)}{q-1},q\right)}$ & 
 $-c$ \\\cline{1-7}
  &   &  \ref{diagconstr} &  $ u \in \mathbb{N}$ & $(b,c)$ & $(n+ u ,q)$ & $\theta+ u (q-1)$\\\cline{3-7}
  &   &  \ref{timestconstr} &  $ v \in \mathbb{N}$ & $( v b, v c)$ & $(n v ,q)$ & $ v \theta$ \\\cline{3-7}
  &   &  \ref{altimesconstr} &  $p \in \mathbb{N}$  & $(pb,pc)$ & $(n,pq)$ & $p\theta+n(p-1)$\\\cline{3-7}
  &   &  \ref{Flaassstand} &   $\exists$ $1$-perfect code in $H(q+1,q)$; $\theta \le 0$; $\tau_1,\tau_2=0,\ldots,q$, $\tau_1+\tau_2\ne 0,2q$ & $(q(b+c)-(c\tau_1+b\tau_2), c\tau_1+b\tau_2)$ & $(qn-\theta,q)$ & $\theta$\\\cline{3-7}
$(b,c)$  & $(n,q)$  &  \ref{Flaassimprov}(1) & $\exists$ $1$-perfect code in $H(q+1,q)$; first color can be partitioned into $k$-faces;  $\theta +k \le 0$;  $ \tau =1,\ldots,q$ & $(q(b+c)- \tau c, \tau c)$ & $(qn-\theta-k,q)$ & $\theta-k(q-1)$\\\cline{3-7}
  &   &  \ref{Flaassimprov}(2) & $\exists$ $1$-perfect code in $H(q+1,q)$;  first color can be partitioned into $k$-faces; $\theta \le k(q-1)$; $ \tau =1,\ldots,q$ & $(q(b+c)- \tau c, \tau c)$ & $(qn-\theta+k(q-1),q)$ & $\theta+k(q-1)^2$\\\cline{1-7}
$(q^2-1,1)$  &  $(q+1,q)$ &  \ref{splitconstr} &  second color can be partitioned into $1$-faces;
$p \in \mathbb{N}$, $0 \leq \tau \leq p-1$ & $((q^2-1)(p- \tau ), (q^2-1) \tau +p)$ & $(q+1,pq)$ & $q(p-1)-1$\\\cline{1-7}
\end{tabular}
\end{center}


\section{Admissible parameters of perfect $2$-colorings of $H(n,q)$} \label{paramsec}

We first mention a well-known corollary from Theorem~\ref{diagconstr}, which implies the following alternative for the existence of perfect $(b,c)$-colorings in Hamming graphs for given $a$ and $b$.

\begin{corollary}
Given  $q$, $b$ and $c$,
either  there are 
no $(b,c)$-colorings 
of $H(n,q)$ 
for all $n \in \mathbb{N}$,
or there is 
$n_0 = n_0 (b,c,q)$ such that $(b,c)$-colorings of $H(n,q)$ exist if and only if $n \ge n_0$.
\end{corollary}

We call parameters $q$, 
$b$ and $c$ 
satisfying the second
alternative 
\emph{admissible}, 
and we will say 
that $n_0 = n_0(b,c;q)$ 
is the \emph{threshold}. 

On the base of above 
theorems and some
computation results, 
we put forward 
the following conjecture.

\begin{conjecture} 
\label{charconj}
Parameters $q$, and $b \geq c\ne 1$
are admissible 
if and only if 
$\frac{b+ c}{\gcd(b,c)}$ divides $q^k$ 
for some $k \in N$.
\end{conjecture}

Necessity of this
condition 
is established 
in Theorem~\ref{corbound}
(by 
Theorem~\ref{c1cond},
it is not sufficient 
if $c=1$).

Combining 
the obtained constructions 
of perfect $2$-colorings,
we confirm 
Conjecture~\ref{charconj}
for some cases 
and provide some bounds 
on the threshold  $n_0$.
We start with some
sufficient conditions 
on admissible parameters 
$(b,c)$ when $q$ 
is a prime power.

Recall that we use notation $b' = \frac{b}{\gcd(b,c)}$ and $c' = \frac{c}{\gcd(b,c)}$.

\begin{theorem} 
\label{combconstr}
Let $q = p^s$ 
be a power of prime $p$. 
If for given $b$ and $c$ we have $b' + c' = q^k$  
for some  
$k \in \mathbb{N}$, 
then the parameters $b$ and $c$
are admissible. 
Moreover, 
 the threshold
 $n_0=n_0(b,c;q)$ 
 satisfies the inequalities 
 $$
 \max \left\{\frac{b}{q-1},\frac{b+c}{q} + k -1\right\} \leq n_0 \leq \frac{b+c -\gcd(b,c)}{q-1}.
 $$
\end{theorem}

\begin{proof}
Note that the parameters 
$b$ and $c$ satisfy 
the necessary condition 
of Theorem~\ref{corbound}.
In addition, 
if $c = 1$ 
then  $b$ and $c$ 
are relatively prime 
and coincide with 
$b'$ and $c'$. 
So, $b = q^k -1$ and 
it is divisible by $q-1$ 
as is claimed in 
Theorem~\ref{c1cond}.

To prove sufficiency,
we need to construct
a $(b,c)$-coloring $f$
in $H(n,q)$ for some $n$.
Denote $m=\gcd(b,c)$,
$b'=b/m$, and $c'=c/m$.
We have $b' + c' = q^k$
and   $b + c = m q^k$ 
for some $k$. 
As it was stated 
in Subsection~\ref{perfcodes},
for $n' = \frac{q^k-1}{q-1}$ 
there is a $c'$-fold 
$1$-perfect code 
in $H(n',q)$,
corresponding 
to a $(b',c')$-coloring $f'$. 
With the help 
of Theorem~\ref{timestconstr},
we multiply 
the perfect coloring $f'$ 
by $m$ and obtain 
a $(b,c)$-coloring 
in the graph $H(n,q)$ with 
$$
n = mn' = m \frac{q^k-1}{q-1} 
 =  \frac{b+c-\gcd(b,c)}{q-1}. 
$$

Let us prove the lower bound 
on the threshold $n_0$. 
Since for every 
$(b,c)$-coloring 
in $H(n,q)$ 
the parameter $b$ 
is not greater than 
the degree $n(q-1)$ 
of the graph $H(n,q)$, 
we have 
$n_0 \geq \frac{b}{q-1}$.  
By Theorem~\ref{corbound}, 
if there exists a 
$(b,c)$-coloring 
in $H(n,q)$, 
then $q^k$ divides 
$q^{n - mq^{k-1} +1}$. 
Since $q$ is a prime power,
we get the inequality  
$n_0 \geq m q^{k-1} + k -1$.
\end{proof}

Additionally, the construction of perfect colorings from additive codes (Section~\ref{s:additive})
enables us to solve the existence of $(b,c)$-colorings in some cases when $q$ 
is a power of prime $p$ 
and $b'+c'$ is also a power of $p$ but not necessarily a power of $q$.

\begin{proposition}\label{pr:spread}
Assume that $s$ and $m$ are integers 
satisfying $1\le s \le m$, $p$ is a prime number,
$q=p^s$, and $l$ is a divisor of both $s$ and $m$.
If $c=\frac{q-1}{p^l-1}$,
 then, 
 $$
 n_0( v c(p^m-1), v c;q) = \frac{ v c(p^m-1)}{q-1}
 \quad\mbox{and}\quad
 n_0( v c(p^m-\tau), v \tau c;q) \le \frac{ v c(p^m-1)}{q-1}$$
 for any $\tau \in \{1, \ldots, p^m-1 \}$ and $ v >0$.
\end{proposition}

\begin{proof}
 By Theorem 4.2.7 \cite{Hirschfeld79} 
 on the existence of multifold spreads,
 the hypothesis of the proposition
 guarantees the existence of a collection of subspaces
 of $\mathbb{Z}_p^m$ satisfying p.(i) 
 in Theorem~\ref{th:additive}.
 By Theorem~\ref{th:additive}, there is
 an additive code corresponding to
 a $((p^m-1)c,c)$-coloring in 
 $H(\frac{(p^m-1)c}{q-1},q)$.
 This gives 
 $ n_0(c(p^m-1), c;q) = \frac{c(p^m-1)}{q-1}$
 (we cannot reduce the dimension because 
 the first element of the quotient
 matrix is $0$)
 and, by~Theorem~\ref{timestconstr}, 
 $ n_0( v  c(p^m-1), v c;q) = \frac{ v c(p^m-1)}{q-1}$.
 Unifying $\tau$ cosets of the additive code and 
 again applying~Theorem~\ref{timestconstr},
 we get a coloring supporting the inequality part of the claim.
\end{proof}

As a direct corollary of Theorem~\ref{combconstr},
we have that Conjecture~\ref{charconj} holds for the case of prime $q$.

\begin{theorem}\label{primeq}
Assume that $q$ is a prime number. 
\begin{enumerate}
\item The parameters $b,c$ and $q$ are admissible if and only if $b' + c' = q^k$ for some $k \in \mathbb{N}$.
\item If  $b' + c' = q^k$ for some $k \in \mathbb{N}$ then the threshold parameter $n_0$ for existence of $(b,c)$-colorings satisfies the inequalities 
 $$\max \left\{\frac{b}{q-1},\frac{b+c}{q} + k -1\right\} \leq n_0 \leq \frac{b+c -\gcd(b,c)}{q-1}.$$
\end{enumerate}
\end{theorem}

\begin{proof}
It is sufficient to note that in case of prime $q$ the  condition from  Theorem~\ref{corbound} ($b' +c'$ divides some power of $q$) is equivalent to that $b' +c'$ is a power of $q$. It only remains to use Theorem~\ref{combconstr}.
\end{proof}

In certain cases, 
we can give 
the exact value 
of the threshold $n_0$:

\begin{corollary}
Let $q$ be a prime.
\begin{enumerate}
\item If $b$ and $c$ 
satisfy $\frac{b + c}{\gcd (b,c)} = q$, then $n_0 = \gcd (b,c)$.
\item If  $b$ and $c$ are relatively prime such that $b+c = q^2$, then $n_0 = q+1$.
\end{enumerate}
\end{corollary}


\section{Conclusion}

 We studied the problem of parameters of perfect
$2$-colorings in Hamming graphs,
considering several known and new constructions and generating
tables of small parameters.
In the conclusion, we highlight one subcase
of the general problem, namely,
the problem of the existence of perfect $2$-colorings of
$H(n,q)$
with quotient matrix
$\left(\begin{matrix}
  a & b \\ c & d
 \end{matrix}
\right)$ satisfying $a = 0$, i.e., when the first color is an independent set.
The existence of such colorings is equivalent to the existence of
an orthogonal array OA$(\frac{q^nc}{b+c},n,q,\frac{b+c}{q}-1)$
attaining the Bierbrauer--Friedman bound.
First examples of such colorings correspond to 
$1$-perfect codes, 
and next we can multiply the quotient matrix by a constant
(the corresponding orthogonal arrays can be constructed as linear codes and are
well known). In the binary case, the main recursive construction
in~\cite{FDF:PerfCol} enables one to reduce
the parameter $a$ in the quotient matrix step by step until $a=0$. 
This gives an infinite series of parameters,
including, for example, $(13,3)$-, $(29,3)$-, $(59,5)$-colorings with $a=0$
(OA$(1536,13,2,7)$, OA$(3\cdot 2^{24},29,2,15)$, OA$(5\cdot 2^{53},59,2,31)$).
In the non-binary case, the analog of the Fon-Der-Flaass construction
(Splitting~I) does not allow us to reduce  $a$, and we can only
construct perfect $2$-colorings with $a=0$ as additive codes.
The additive construction gives new parameters for nonprime prime-power $q$,
but for odd prime $q$ the question of the existence 
of $2$-colorings in $H(n,q)$ with $a=0$ 
that do not come from perfect codes is open.
The only parameters with $a=0$ for which the nonexistence of colorings is known
 are the parameters of $(14,4)$-colorings in $H(7,3)$,
corresponding to OA$(2\cdot 3^5,7,3,5)$ \cite{HSS:97b}.
In this direction, the first open questions are about 
$(25,7)$-colorings in $H(25,2)$,
$(22,5)$-colorings in $H(11,3)$, $(39,9)$-colorings in $H(13,4)$.


\section*{Acknowledgements}
This work was funded by 
the Russian Science Foundation 
under grant 18-11-00136
(Sections~1--A.2, except Theorems~\ref{Flaassstand}
and~\ref{Flaassimprov}(1), 
whose early versions 
occurred as a part of 
the graduate thesis of Aleksandr Matiushev
in the Novosibirsk State University);
the work of Anna Taranenko was 
supported in part (Section~A.3) by an award of  the contest 
``Young Russian Mathematics''.


\bibliography{k}


\newpage
 
\appendix
 
\section{Appendix. Perfect colorings in Hamming graphs of small sizes}

We use the following notations for the columns of further tables

\textbf{Lower bounds:}

$a$ -- degree bound $a \geq 0$: $n =\left\lceil \frac{b}{q-1} \right\rceil $.

$k$ -- bound  from Theorem~\ref{corbound}: $n = \frac{b+c}{q} + k -1$, where $k$ is the minimal integer such that $\frac{b+c}{\gcd(b,c)}$ divides $q^k$.

\textbf{Constructions:}

$*$ -- multiplication of length 
(Theorem~\ref{timestconstr}).

$q$ -- multiplication of alphabet 
(Theorem~\ref{altimesconstr}).

$P$ -- $1$-perfect codes, multifold $1$-perfect codes, additive codes
(Sections~\ref{s:additive} and~\ref{perfcodes}).

$F$ -- best of the splitting-I constructions (Theorems~\ref{Flaassstand} and \ref {Flaassimprov}).

$S$ -- multiplication of alphabet 
of a $1$-perfect code in $H(q+1,q)$ 
with splitting of the second color 
(splitting-II, Theorem~\ref{splitconstr}).

\medskip 

The best lower bounds and constructions for given parameters $(b,c)$ are bold in each row. 
If there is a gap between lower bounds and constructions for a $(b,c)$-coloring, then the corresponding row is highlighted by color. 
The sign ``?'' in the last column means that  no $(b,c)$-colorings
are known
in $H(n,q)$ for any $n$.


\subsection{Case $q=3$, illustrating prime $q$}\label{a:q3} 
\def\ssz{\footnotesize}
\newcommand\hlt{\rowcolor{LightCyan}}
\newcommand\hltc{\cellcolor{LightCyan}}
\definecolor{LightCyan}{rgb}{0.88,1,1}
\begin{center}
\begin{tabular}{|c|c|c||rr |r||rrr|r||}\cline{1-10}
$b{+c}$ & $b' {+}c'$ & $(b,c)$  & a & k & \bf \!\!LB& * & P & F & \bf \!\!UB   \\ \hhline{|=|=|=#==|=#===|=||}
3  &  3 &   (2,1) & \ssz\bf{ 1} & \ssz\bf{ 1} & \bf  1       & \ssz    --  & \ssz \bf{1} & \ssz     -- & \bf 1  \\\cline{1-10}
6  &  3 &   (4,2) & \ssz\bf{ 2} & \ssz\bf{ 2} & \bf  2       & \ssz\bf{2}  & \ssz     -- & \ssz    --  & \bf 2 \\\cline{1-10}
9  &  3 &   (6,3) & \ssz\bf{ 3} & \ssz\bf{ 3} & \bf  3       & \ssz\bf{3}  & \ssz      4 & \ssz      4 & \bf 3  \\
9  &  9 &   (5,4) & \ssz     3  & \ssz\bf{ 4} & \bf  4       & \ssz    --  & \ssz \bf{4} & \ssz\bf{ 4} & \bf 4 \\
9  &  9 &   (7,2) & \ssz\bf{ 4} & \ssz\bf{ 4} & \bf  4       & \ssz    --  & \ssz \bf{4} & \ssz\bf{ 4} & \bf 4 \\
9  &  9 &   (8,1) & \ssz\bf{ 4} & \ssz\bf{ 4} & \bf  4       & \ssz    --  & \ssz \bf{4} & \ssz\bf{ 4} & \bf 4  \\\cline{1-10}
12 &  3 &   (8,4) & \ssz\bf{ 4} & \ssz\bf{ 4} & \bf  4       & \ssz\bf{ 4} & \ssz     -- & \ssz     -- & \bf  4 \\\cline{1-10}
15 &  3 &  (10,5) & \ssz\bf{ 5} & \ssz\bf{ 5} & \bf  5       & \ssz\bf{ 5} & \ssz     -- & \ssz     -- & \bf  5 \\\cline{1-10}
18 &  3 &  (12,6) & \ssz\bf{ 6} & \ssz\bf{ 6} & \bf  6       & \ssz\bf{ 6} & \ssz     -- & \ssz      8 & \bf  6 \\
\hlt 18 &  9 &  (10,8) & \ssz     5  & \ssz\bf{ 7} & \bf  7       & \ssz\bf{ 8} & \ssz     -- & \ssz\bf{ 8} & \bf  8 \\
18 &  9 &  (14,4) & \ssz     7  & \ssz     7  & \bf  8$^\dag$& \ssz\bf{ 8} & \ssz     -- & \ssz\bf{ 8} & \bf  8 \\
18 &  9 &  (16,2) & \ssz\bf{ 8} & \ssz     7  & \bf  8       & \ssz\bf{ 8} & \ssz     -- & \ssz\bf{ 8} & \bf  8\\\cline{1-10}
21 &  3 &  (14,7) & \ssz\bf{ 7} & \ssz\bf{ 7} & \bf  7       & \ssz\bf{ 7} & \ssz     -- & \ssz     -- & \bf  7 \\\cline{1-10}
24 &  3 &  (16,8) & \ssz\bf{ 8} & \ssz\bf{ 8} & \bf  8       & \ssz\bf{ 8} & \ssz     -- & \ssz     -- & \bf  8 \\\cline{1-10}
27 &  3 &  (18,9) & \ssz\bf{ 9} & \ssz\bf{ 9} & \bf  9       & \ssz\bf{ 9} & \ssz     13 & \ssz    12  & \bf  9 \\
\hlt 27 &  9 & (15,12) & \ssz     8  & \ssz\bf{10} & \bf 10       & \ssz\bf{12} & \ssz     13 & \ssz\bf{12} & \bf 12 \\
\hlt 27 &  9 &  (21,6) & \ssz\bf{11} & \ssz    10  & \bf 11       & \ssz\bf{12} & \ssz     13 & \ssz\bf{12} & \bf 12 \\
27 &  9 &  (24,3) & \ssz\bf{12} & \ssz    10  & \bf 12       & \ssz\bf{12} & \ssz     13 & \ssz\bf{12} & \bf 12 \\
\hlt 27 & 27 & (14,13) & \ssz     7  & \ssz\bf{11} & \bf 11       & \ssz     -- & \ssz\bf{13} & \ssz\bf{13} & \bf 13 \\
\hlt 27 & 27 & (16,11) & \ssz     8  & \ssz\bf{11} & \bf 11       & \ssz     -- & \ssz     13 & \ssz\bf{12} & \bf 12 \\
\hlt 27 & 27 & (17,10) & \ssz     9  & \ssz\bf{11} & \bf 11       & \ssz     -- & \ssz\bf{13} & \ssz\bf{13} & \bf 13 \\
\hlt 27 & 27 &  (19,8) & \ssz    10  & \ssz\bf{11} & \bf 11       & \ssz     -- & \ssz     13 & \ssz\bf{12} & \bf 12 \\
\hlt 27 & 27 &  (20,7) & \ssz    10  & \ssz\bf{11} & \bf 11       & \ssz     -- & \ssz\bf{13} & \ssz\bf{13} & \bf 13 \\
\hlt 27 & 27 &  (22,5) & \ssz\bf{11} & \ssz\bf{11} & \bf 11       & \ssz     -- & \ssz\bf{13} & \ssz\bf{13} & \bf 13 \\
\hlt 27 & 27 &  (23,4) & \ssz\bf{12} & \ssz    11  & \bf 12       & \ssz     -- & \ssz\bf{13} & \ssz\bf{13} & \bf 13 \\
27 & 27 &  (25,2) & \ssz\bf{13} & \ssz    11  & \bf 13       & \ssz     -- & \ssz\bf{13} & \ssz\bf{13} & \bf 13 \\
27 & 27 &  (26,1) & \ssz\bf{13} & \ssz    11  & \bf 13       & \ssz     -- & \ssz\bf{13} & \ssz\bf{13} & \bf 13 \\\cline{1-10}
30 &  3 & (20,10) & \ssz\bf{10} & \ssz\bf{10} & \bf 10       & \ssz\bf{10} & \ssz     -- & \ssz     -- & \bf 10 \\\cline{1-10}
33 &  3 & (22,11) & \ssz\bf{11} & \ssz\bf{11} & \bf 11       & \ssz\bf{11} & \ssz     -- & \ssz     -- & \bf 11 \\\cline{1-10}
36 &  3 & (24,12) & \ssz\bf{12} & \ssz\bf{12} & \bf 12       & \ssz\bf{12} & \ssz     -- & \ssz     16 & \bf 12 \\
\hlt 36 &  9 & (20,16) & \ssz    10  & \ssz\bf{13} & \bf 13       & \ssz\bf{16} & \ssz     -- & \ssz\bf{16} & \bf 16 \\
\hlt 36 &  9 & (28,8)  & \ssz\bf{14} & \ssz    13  & \bf 14       & \ssz\bf{16} & \ssz     -- & \ssz\bf{16} & \bf 16 \\
36 &  9 & (32,4)  & \ssz\bf{16} & \ssz    13  & \bf 16       & \ssz\bf{16} & \ssz     -- & \ssz\bf{16} & \bf 16 \\\cline{1-10}
39 &  3 & (26,13) & \ssz\bf{13} & \ssz\bf{13} & \bf 13       & \ssz\bf{13} & \ssz     -- & \ssz     -- & \bf 13 \\\cline{1-10}
42 &  3 & (28,14) & \ssz\bf{14} & \ssz\bf{14} & \bf 14       & \ssz\bf{14} & \ssz     -- & \ssz     -- & \bf 14 \\\cline{1-10}
45 &  3 & (30,15) & \ssz\bf{15} & \ssz\bf{15} & \bf 15       & \ssz\bf{15} & \ssz     -- & \ssz     20 & \bf 15 \\
\hlt 45 &  9 & (25,20) & \ssz    13  & \ssz\bf{16} & \bf 16       & \ssz\bf{20} & \ssz     -- & \ssz\bf{20} & \bf 20 \\
\hlt 45 &  9 & (35,10) & \ssz\bf{18} & \ssz    16  & \bf 18       & \ssz\bf{20} & \ssz     -- & \ssz\bf{20} & \bf 20 \\
45 &  9 & (40,5)  & \ssz\bf{20} & \ssz    16  & \bf 20       & \ssz\bf{20} & \ssz     -- & \ssz\bf{20} & \bf 20 \\\cline{1-10}
48 &  3 & (32,16) & \ssz\bf{16} & \ssz\bf{16} & \bf 16       & \ssz\bf{16} & \ssz     -- & \ssz     -- & \bf 16 \\\cline{1-10}
51 &  3 & (34,17) & \ssz\bf{17} & \ssz\bf{17} & \bf 17       & \ssz\bf{17} & \ssz     -- & \ssz     -- & \bf 17 \\\cline{1-10}
54 &  3 & (36,18) & \ssz\bf{18} & \ssz\bf{18} & \bf 18       & \ssz\bf{18} & \ssz     -- & \ssz     24 & \bf 18 \\
\hlt 54 &  9 & (30,24) & \ssz    15  & \ssz\bf{19} & \bf 19       & \ssz\bf{24} & \ssz     -- & \ssz\bf{24} & \bf 24 \\
\hlt 54 &  9 & (42,12) & \ssz\bf{21} & \ssz    19  & \bf 21       & \ssz\bf{24} & \ssz     -- & \ssz\bf{24} & \bf 24 \\
54 &  9 & (48,6)  & \ssz\bf{24} & \ssz    19  & \bf 24       & \ssz\bf{24} & \ssz     -- & \ssz\bf{24} & \bf 24 \\
\hlt 54 & 27 & (28,26) & \ssz    14  & \ssz\bf{20} & \bf 20       & \ssz\bf{26} & \ssz     -- & \ssz\bf{26} & \bf 26 \\
\hlt 54 & 27 & (34,20) & \ssz    17  & \ssz\bf{20} & \bf 20       & \ssz\bf{26} & \ssz     -- & \ssz\bf{26} & \bf 26 \\
\hlt 54 & 27 & (38,16) & \ssz    19  & \ssz\bf{20} & \bf 20       & \ssz\bf{24} & \ssz     -- & \ssz     26 & \bf 24 \\
\hlt 54 & 27 & (40,14) & \ssz\bf{20} & \ssz\bf{20} & \bf 20       & \ssz\bf{26} & \ssz     -- & \ssz\bf{26} & \bf 26 \\
\hlt 54 & 27 & (44,10) & \ssz\bf{22} & \ssz    20  & \bf 22       & \ssz\bf{26} & \ssz     -- & \ssz\bf{26} & \bf 26 \\
\hlt 54 & 27 &  (46,8) & \ssz\bf{23} & \ssz    20  & \bf 23       & \ssz\bf{26} & \ssz     -- & \ssz\bf{26} & \bf 26 \\
\hlt 54 & 27 &  (50,4) & \ssz\bf{25} & \ssz    20  & \bf 25       & \ssz\bf{26} & \ssz     -- & \ssz\bf{26} & \bf 26 \\
54 & 27 &  (52,2) & \ssz\bf{26} & \ssz    20  & \bf 26       & \ssz\bf{26} & \ssz     -- & \ssz\bf{26} & \bf 26 \\\cline{1-10}
\end{tabular}
\end{center}

\dag : $(14,4)$-colorings in $H(7,3)$ do not exist because there are no OA$(2\cdot 3^5,7,3,5)$~\cite{HSS:97b}.

\newpage
\subsection{Case $q=4$, illustrating a prime power $q$}\label{a:q4}

\begin{center}
\begin{tabular}{|c|c|c||rr |r||rrrrr|r||} \cline{1-12}
$b+c$ & $b' +c'$ & $(b,c)$  & a & k & \bf \!\!LB& * & q & P & F & S  & \bf \!\!UB  \\ \hhline{|=|=|=#==|=#=====|=#}
 4 &  2 & ( 2, 2) & \ssz\bf{ 1} & \ssz\bf{ 1} & \bf  1 & \ssz    --  & \ssz\bf{ 1} & \ssz\bf{1} & \ssz    --  & \ssz   -- & \bf 1   \\
 4 &  4 & ( 3, 1) & \ssz\bf{ 1} & \ssz\bf{ 1} & \bf  1 & \ssz    --  & \ssz    --  & \ssz\bf{1} & \ssz    --  & \ssz   -- & \bf 1   \\\cline{1-12}
 8 &  2 & ( 4, 4) & \ssz\bf{ 2} & \ssz\bf{ 2} & \bf  2 & \ssz\bf{ 2} & \ssz\bf{ 2} & \ssz   --  & \ssz    --  & \ssz   -- & \bf 2   \\
 8 &  4 & ( 6, 2) & \ssz\bf{ 2} & \ssz\bf{ 2} & \bf  2 & \ssz\bf{ 2} & \ssz     3  & \ssz   --  & \ssz    --  & \ssz   -- & \bf 2   \\
 8 &  8 & ( 5, 3) & \ssz     2  & \ssz\bf{ 3} & \bf  3 & \ssz    --  & \ssz    --  & \ssz   --  & \ssz    --  & \ssz\bf{3}& \bf 3   \\\cline{1-12}
12 &  2 & ( 6, 6) & \ssz     2  & \ssz\bf{ 3} & \bf  3 & \ssz\bf{ 3} & \ssz\bf{ 3} & \ssz   --  & \ssz    --  & \ssz   -- & \bf 3   \\
12 &  4 & ( 9, 3) & \ssz\bf{ 3} & \ssz\bf{ 3} & \bf  3 & \ssz\bf{ 3} & \ssz    --  & \ssz   --  & \ssz    --  & \ssz   -- & \bf 3   \\\cline{1-12}
16 &  2 & ( 8, 8) & \ssz     3  & \ssz\bf{ 4} & \bf  4 & \ssz\bf{ 4} & \ssz\bf{ 4} & \ssz    5  & \ssz     5  & \ssz   -- & \bf 4   \\
16 &  4 & (12, 4) & \ssz\bf{ 4} & \ssz\bf{ 4} & \bf  4 & \ssz\bf{ 4} & \ssz     6  & \ssz    5  & \ssz     5  & \ssz   -- & \bf 4   \\
16 &  8 & (10, 6) & \ssz     4  & \ssz\bf{ 5} & \bf  5 & \ssz     6  & \ssz     6  & \ssz\bf{5} & \ssz\bf{ 5} & \ssz   -- & \bf 5   \\
16 &  8 & (14, 2) & \ssz\bf{ 5} & \ssz\bf{ 5} & \bf  5 & \ssz    --  & \ssz     7  & \ssz\bf{5} & \ssz\bf{ 5} & \ssz   -- & \bf 5   \\
16 & 16 & ( 9, 7) & \ssz     3  & \ssz\bf{ 5} & \bf  5 & \ssz    --  & \ssz    --  & \ssz\bf{5} & \ssz\bf{ 5} & \ssz   -- & \bf 5   \\
16 & 16 & (11, 5) & \ssz     4  & \ssz\bf{ 5} & \bf  5 & \ssz    --  & \ssz    --  & \ssz\bf{5} & \ssz\bf{ 5} & \ssz   -- & \bf 5   \\
16 & 16 & (13, 3) & \ssz\bf{ 5} & \ssz\bf{ 5} & \bf  5 & \ssz    --  & \ssz    --  & \ssz\bf{5} & \ssz\bf{ 5} & \ssz   -- & \bf 5   \\
16 & 16 & (15, 1) & \ssz\bf{ 5} & \ssz\bf{ 5} & \bf  5 & \ssz    --  & \ssz    --  & \ssz\bf{5} & \ssz\bf{ 5} & \ssz   -- & \bf 5   \\\cline{1-12}
20 &  2 & (10,10) & \ssz     4  & \ssz\bf{ 5} & \bf  5 & \ssz\bf{ 5} & \ssz\bf{ 5} & \ssz   --  & \ssz    --  & \ssz   -- & \bf 5   \\
20 &  4 & (15, 5) & \ssz     5  & \ssz\bf{ 5} & \bf  5 & \ssz\bf{ 5} & \ssz    --  & \ssz   --  & \ssz    --  & \ssz   -- & \bf 5   \\\cline{1-12}
24 &  2 & (12,12) & \ssz     4  & \ssz\bf{ 6} & \bf  6 & \ssz\bf{ 6} & \ssz\bf{ 6} & \ssz   --  & \ssz    --  & \ssz   -- & \bf 6   \\
24 &  4 & (18, 6) & \ssz\bf{ 6} & \ssz\bf{ 6} & \bf  6 & \ssz\bf{ 6} & \ssz    --  & \ssz   --  & \ssz    --  & \ssz   -- & \bf 6   \\
 24 &  8 & (15, 9) & \ssz     5  & \ssz\bf{ 7} & \bf  7 & \ssz\bf{ 9} & \ssz    --  & \ssz  \bf{{7}}  & \ssz    --  & \ssz   -- & \bf 7  \\
 24 &  8 & (21, 3) & \ssz\bf{ 7} & \ssz\bf{ 7} & \bf  7 & \ssz    --  & \ssz    --  & \ssz   \bf{{7}}  & \ssz    --  & \ssz   -- & \bf 7  \\\cline{1-12}
28 &  2 & (14,14) & \ssz     5  & \ssz\bf{ 7} & \bf  7 & \ssz\bf{ 7} & \ssz\bf{ 7} & \ssz   --  & \ssz    --  & \ssz   -- & \bf 7   \\
28 &  4 & (21, 7) & \ssz\bf{ 7} & \ssz\bf{ 7} & \bf  7 & \ssz\bf{ 7} & \ssz    --  & \ssz   --  & \ssz    --  & \ssz   -- & \bf 7   \\\cline{1-12}
32 &  2 & (16,16) & \ssz     6  & \ssz\bf{ 8} & \bf  8 & \ssz\bf{ 8} & \ssz\bf{ 8} & \ssz   --  & \ssz    10  & \ssz   -- & \bf 8   \\
32 &  4 & (24, 8) & \ssz\bf{ 8} & \ssz\bf{ 8} & \bf  8 & \ssz\bf{ 8} & \ssz    12  & \ssz   --  & \ssz    10  & \ssz   -- & \bf 8   \\
\hlt 32 &  8 & (20,12) & \ssz     7  & \ssz\bf{ 9} & \bf  9 & \ssz\bf{10} & \ssz    12  & \ssz   --  & \ssz\bf{10} & \ssz   -- & \bf10  \\
32 &  8 & (28, 4) & \ssz\bf{10} & \ssz     9  & \bf 10 & \ssz\bf{10} & \ssz    14  & \ssz   --  & \ssz\bf{10} & \ssz   -- & \bf10   \\
\hlt 32 & 16 & (18,14) & \ssz     6  & \ssz\bf{ 9} & \bf  9 & \ssz\bf{10} & \ssz    12  & \ssz   --  & \ssz\bf{10} & \ssz   -- & \bf10  \\
\hlt 32 & 16 & (22,10) & \ssz     8  & \ssz\bf{ 9} & \bf  9 & \ssz\bf{10} & \ssz    13  & \ssz   --  & \ssz\bf{10} & \ssz   -- & \bf10  \\
\hlt 32 & 16 & (26, 6) & \ssz \bf{9} & \ssz\bf{ 9} & \bf  9 & \ssz\bf{10} & \ssz    13  & \ssz   --  & \ssz\bf{10} & \ssz   -- & \bf10  \\
32 & 16 & (30, 2) & \ssz\bf{10} & \ssz     9  & \bf 10 & \ssz\bf{10} & \ssz    15  & \ssz   --  & \ssz\bf{10} & \ssz   -- & \bf10   \\
\hlt 32 & 32 & (17,15) & \ssz     6  & \ssz\bf{10} & \bf 10 & \ssz    --  & \ssz    --  & \ssz   --  & \ssz    --  & \ssz   -- &     ?  \\
\hlt 32 & 32 & (19,13) & \ssz     7  & \ssz\bf{10} & \bf 10 & \ssz    --  & \ssz    --  & \ssz   --  & \ssz    --  & \ssz   -- &     ?  \\
\hlt 32 & 32 & (21,11) & \ssz     7  & \ssz\bf{10} & \bf 10 & \ssz    --  & \ssz    --  & \ssz   --  & \ssz    --  & \ssz   -- &     ?  \\
\hlt 32 & 32 & (23, 9) & \ssz     8  & \ssz\bf{10} & \bf 10 & \ssz    --  & \ssz    --  & \ssz   --  & \ssz\bf{14} & \ssz   -- & \bf14  \\
\hlt 32 & 32 & (25, 7) & \ssz     9  & \ssz\bf{10} & \bf 10 & \ssz    --  & \ssz    --  & \ssz   --  & \ssz    --  & \ssz   -- &     ?  \\
\hlt 32 & 32 & (27, 5) & \ssz     9  & \ssz\bf{10} & \bf 10 & \ssz    --  & \ssz    --  & \ssz   --  & \ssz    --  & \ssz   -- &     ?  \\
\hlt 32 & 32 & (29, 3) & \ssz\bf{10} & \ssz\bf{10} & \bf 10 & \ssz    --  & \ssz    --  & \ssz   --  & \ssz\bf{14} & \ssz   -- & \bf14  \\\cline{1-12}
36 &  2 & (18,18) & \ssz     6  & \ssz\bf{ 9} & \bf  9 & \ssz\bf{ 9} & \ssz\bf{ 9} & \ssz   --  & \ssz    --  & \ssz   -- & \bf 9   \\
36 &  4 & (27, 9) & \ssz\bf{ 9} & \ssz\bf{ 9} & \bf  9 & \ssz\bf{ 9} & \ssz    --  & \ssz   --  & \ssz    --  & \ssz   -- & \bf 9   \\\cline{1-12}
40 &  2 & (20,20) & \ssz     7  & \ssz\bf{10} & \bf 10 & \ssz\bf{10} & \ssz\bf{10} & \ssz   --  & \ssz    --  & \ssz   -- & \bf10   \\
40 &  4 & (30,10) & \ssz\bf{10} & \ssz\bf{10} & \bf 10 & \ssz\bf{10} & \ssz    15  & \ssz   --  & \ssz    --  & \ssz   -- & \bf10   \\
\hlt 40 &  8 & (25,15) & \ssz     9  & \ssz\bf{11} & \bf 11 & \ssz\bf{15} & \ssz    --  & \ssz   \bf{{12}}  & \ssz    --  & \ssz   -- & \bf12  \\
 40 &  8 & (35, 5) & \ssz\bf{12} & \ssz    11  & \bf 12 & \ssz    --  & \ssz    --  & \ssz   \bf{{12}}  & \ssz    --  & \ssz   -- &    \bf12   \\\cline{1-12}
44 &  2 & (22,22) & \ssz     8  & \ssz\bf{11} & \bf 11 & \ssz\bf{11} & \ssz\bf{11} & \ssz   --  & \ssz    --  & \ssz   -- & \bf11   \\
44 &  4 & (33,11) & \ssz\bf{11} & \ssz\bf{11} & \bf 11 & \ssz\bf{11} & \ssz    --  & \ssz   --  & \ssz    --  & \ssz   -- & \bf11   \\\cline{1-12}
48 &  2 & (24,24) & \ssz     8  & \ssz\bf{12} & \bf 12 & \ssz\bf{12} & \ssz\bf{12} & \ssz   --  & \ssz    15  & \ssz   -- & \bf12   \\
48 &  4 & (36,12) & \ssz\bf{12} & \ssz\bf{12} & \bf 12 & \ssz\bf{12} & \ssz    18  & \ssz   --  & \ssz    15  & \ssz   -- & \bf12   \\
\hlt 48 &  8 & (30,18) & \ssz    10  & \ssz\bf{13} & \bf 13 & \ssz\bf{14} & \ssz    18  & \ssz   \bf{{14}}  & \ssz\bf{15} & \ssz   -- & \bf14  \\
 48 &  8 & (42, 6) & \ssz\bf{14} & \ssz    13  & \bf 14 & \ssz\bf{14} & \ssz    21  & \ssz   \bf{{14}}  & \ssz\bf{15} & \ssz   -- & \bf14  \\
\hlt 48 & 16 & (27,21) & \ssz     9  & \ssz\bf{13} & \bf 13 & \ssz\bf{15} & \ssz    --  & \ssz   --  & \ssz\bf{15} & \ssz   -- & \bf15  \\
\hlt 48 & 16 & (33,15) & \ssz    11  & \ssz\bf{13} & \bf 13 & \ssz\bf{15} & \ssz    --  & \ssz   --  & \ssz\bf{15} & \ssz   -- & \bf15  \\
\hlt 48 & 16 & (39, 9) & \ssz\bf{13} & \ssz\bf{13} & \bf 13 & \ssz\bf{15} & \ssz    --  & \ssz   --  & \ssz\bf{15} & \ssz   -- & \bf15  \\
48 & 16 & (45, 3) & \ssz\bf{15} & \ssz    13  & \bf 15 & \ssz\bf{15} & \ssz    --  & \ssz   --  & \ssz\bf{15} & \ssz   -- & \bf15   \\   \cline{1-12}
\end{tabular}
\end{center}

\newpage
\subsection{Case $q=6$, illustrating $q$ which is not a prime power}\label{a:q6}
 \begin{center}
\begin{tabular}{|c|c|c||rr |r||rrrr|r||}\cline{1-11}
$b+c$	& $b' +c'$	& $(b,c)$ 	& a	& k	& \bf \!\!LB& *	& q	& P	& S  	& \bf \!\!UB \\ \hhline{|=|=|=#==|=#====|=#}
6	& 2	& (3,3) 	&  \ssz \bf{1}	& \ssz\bf{1}	& \bf{1}	&\ssz --	&\ssz \bf{1}	& \ssz\bf{1}	& \ssz --	&     \bf{1}	 \\
6	& 3	& (4,2) 	&  \ssz \bf{1}	& \ssz\bf{1}	& \bf{1}	&\ssz --	&\ssz \bf{1}	& \ssz\bf{1}	& \ssz -- 	&     \bf{1} 	 \\
6	& 6	& (5,1) 	&  \ssz \bf{1}	& \ssz\bf{1}	& \bf{1}	&\ssz --	&\ssz --	& \ssz\bf{1}	& \ssz -- 	&     \bf{1} 	 \\\cline{1-11}
12	& 2	& (6,6) 	&  \ssz \bf{2}	& \ssz\bf{2}	& \bf{2}	&\ssz \bf{2}	&\ssz \bf{2}	& \ssz    --	& \ssz 3	&     \bf{2}	  \\ \cline{1-11}
12	& 3	& (8,4) 	&  \ssz \bf{2}	& \ssz\bf{2}	& \bf{2}	&\ssz \bf{2}	&\ssz \bf{2}	& \ssz    --	& \ssz -- 	&     \bf{2}	 \\
12	& 4	& (9,3) 	&  \ssz 2	& \ssz\bf{3}	& \bf{3}	&\ssz --	&\ssz \bf{3}	& \ssz    --	& \ssz\bf{3}	&     \bf{3}	 \\
12	& 6	& (10,2)	&  \ssz \bf{2}	& \ssz\bf{2}	& \bf{2}	&\ssz \bf{2}	&\ssz --	& \ssz    --	& \ssz --	&     \bf{2}	 \\
\hlt 12	& 12	& (7,5) 	&  \ssz 2	& \ssz\bf{3}	& \bf{3}	&\ssz --    	&\ssz --	& \ssz    --	& \ssz --	&      ?    	 \\\cline{1-11}
18	& 2	& (9,9) 	&  \ssz 2	& \ssz\bf{3}	& \bf{3}	&\ssz \bf{3}	&\ssz --	& \ssz    --	& \ssz -- 	&     \bf{3}	 \\
18	& 3	& (12,6)	&  \ssz \bf{3}	& \ssz\bf{3}	& \bf{3}	&\ssz \bf{3}	&\ssz \bf{3}	& \ssz    --	& \ssz --	&     \bf{3}	 \\
18	& 6	& (15,3)	&  \ssz \bf{3}	& \ssz\bf{3}	& \bf{3}	&\ssz \bf{3}	&\ssz --	& \ssz    --	& \ssz --	&     \bf{3}	 \\
18	& 9	& (10,8)	&  \ssz 2	& \ssz\bf{4}	& \bf{4}	&\ssz --	&\ssz \bf{4}	& \ssz    --	& \ssz\bf{4}	&     \bf{4}	 \\
18	& 9	& (14,4) 	&  \ssz 3	& \ssz\bf{4}	& \bf{4}	&\ssz --	&\ssz \bf{4}	& \ssz    --	& \ssz -- 	&      \bf{4} 	 \\
18	& 9	& (16,2) 	&  \ssz \bf{4}	& \ssz\bf{4}	& \bf{4}	&\ssz --	&\ssz \bf{4}	& \ssz    --	& \ssz -- 	&     \bf{4}	 \\
\hlt 18	& 18	& (11,7)	&  \ssz 3	& \ssz\bf{4}	& \bf{4}	&\ssz --	&\ssz --	& \ssz    --	& \ssz -- 	&     ?	 \\
\hlt 18	& 18	& (13,5)	&  \ssz 3	& \ssz\bf{4}	& \bf{4}	&\ssz --	&\ssz --	& \ssz    --	& \ssz -- 	&     ?	 \\\cline{1-11}
24	& 2	& (12,12)	&  \ssz 3	& \ssz\bf{4}	& \bf{4}	&\ssz \bf{4}	&\ssz \bf{4}	& \ssz    --	& \ssz -- 	&     \bf{4}	 \\
24	& 3	& (16,8)	&  \ssz \bf{4}	& \ssz\bf{4}	& \bf{4}	&\ssz \bf{4}	&\ssz \bf{4}	& \ssz    --	& \ssz -- 	&     \bf{4}	 \\
\hlt 24	& 4	& (18,6)	&  \ssz 4	& \ssz\bf{5}	& \bf{5}	&\ssz \bf{6}	&\ssz \bf{6}	& \ssz    --	& \ssz --	&     \bf{6}	 \\
\hlt 24	& 6	& (20,4)	&  \ssz \bf{4}	& \ssz\bf{4}	& \bf{4}	&\ssz \bf{5}	&\ssz --	& \ssz    --	& \ssz --	&     \bf{5}	 \\
24	& 8	& (15,9)	&  \ssz 3	& \ssz\bf{6}	& \bf{6}	&\ssz --	&\ssz \bf{6}	& \ssz    --	& \ssz --	&     \bf{6}	 \\
\hlt 24	& 8	& (21,3)	&  \ssz 5	& \ssz\bf{6}	& \bf{6}	&\ssz --	&\ssz \bf{7}	& \ssz    --	& \ssz --	&     \bf{7}	 \\
\hlt 24	& 12	& (14,10)	&  \ssz 3	& \ssz\bf{5}	& \bf{5}	&\ssz --	&\ssz --	& \ssz    --	& \ssz -- 	&      ? 	 \\
\hlt 24	& 12	& (22,2)	&  \ssz \bf{5}	& \ssz\bf{5}	& \bf{5}	&\ssz --	&\ssz --	& \ssz    --	& \ssz -- 	&      ? 	 \\
\hlt 24	& 24	& (13,11)	&  \ssz 3	& \ssz\bf{6}	& \bf{6}	&\ssz --	&\ssz --	& \ssz    --	& \ssz --	&      ?	 \\
\hlt 24	& 24	& (17,7)	&  \ssz 4	& \ssz\bf{6}	& \bf{6}	&\ssz --	&\ssz --	& \ssz    --	& \ssz --	&      ?	 \\
\hlt 24	& 24	& (19,5)	&  \ssz 4	& \ssz\bf{6} 	& \bf{6} 	&\ssz --	&\ssz --	& \ssz    --	& \ssz --	&      ?	 \\\cline{1-11}
30	& 2	& (15,15)	&  \ssz 3	& \ssz\bf{5}	& \bf{5}	&\ssz \bf{5}	&\ssz \bf{5}	& \ssz    --	& \ssz --	&      \bf{5}	 \\
30	& 3	& (20,10)	&  \ssz 4	& \ssz\bf{5}	& \bf{5}	&\ssz \bf{5}	&\ssz \bf{5}	& \ssz    --	& \ssz --	&      \bf{5}	 \\
30	& 6	& (25,5)	&  \ssz \bf{5}	& \ssz\bf{5}	& \bf{5}	&\ssz \bf{5}	&\ssz --	& \ssz    --	& \ssz --	&      \bf{5}	 \\\cline{1-11}
36	& 2	& (18,18)	&  \ssz 4	& \ssz\bf{6}	& \bf{6}	&\ssz \bf{6}	&\ssz 9 	& \ssz    --	& \ssz --	&      \bf{6}	 \\
36	& 3	& (24,12)	&  \ssz 5	& \ssz\bf{6}	& \bf{6}	&\ssz \bf{6}	&\ssz \bf{6}	& \ssz    --	& \ssz -- 	&      \bf{6} 	 \\
\hlt 36	& 4	& (27,9)	&  \ssz 6	& \ssz\bf{7}	& \bf{7}	&\ssz \bf{9}	&\ssz \bf{9}	& \ssz    --	& \ssz --	&      \bf{9}	 \\
36	& 6	& (30,6)	&  \ssz \bf{6}	& \ssz\bf{6}	& \bf{6}	&\ssz \bf{6}	&\ssz --	& \ssz    --	& \ssz --	&      \bf{6}	 \\
\hlt 36	& 9	& (20,16)	&  \ssz 4	& \ssz\bf{7}	& \bf{7}	&\ssz \bf{8}	&\ssz \bf{8}	& \ssz    --	& \ssz --	&      \bf{8}	 \\
\hlt 36	& 9	& (28,8)	&  \ssz 6	& \ssz\bf{7}	& \bf{7}	&\ssz \bf{8}	&\ssz \bf{8}	& \ssz    --	& \ssz --	&      \bf{8}	 \\
\hlt 36	& 9	& (32,4)	&  \ssz \bf{7}	& \ssz\bf{7}	& \bf{7}	&\ssz \bf{8}	&\ssz \bf{8}	& \ssz    --	& \ssz --	&      \bf{8}	 \\
\hlt 36	& 12	& (21,15)	&  \ssz 5	& \ssz\bf{7}	& \bf{7}	&\ssz --	&\ssz --	& \ssz    --	& \ssz --	&      ?	 \\
\hlt 36	& 12	& (33,3)	&  \ssz \bf{7}	& \ssz\bf{7}	& \bf{7}	&\ssz --	&\ssz --	& \ssz    --	& \ssz --	&      ?	 \\
\hlt 36	& 18	& (22,14)	&  \ssz 5	& \ssz\bf{7}	& \bf{7}	&\ssz --	&\ssz --	& \ssz    --	& \ssz --	&      ?	 \\
\hlt 36	& 18	& (26,10)	&  \ssz 6	& \ssz\bf{7}	& \bf{7}	&\ssz --	&\ssz --	& \ssz    --	& \ssz --	&      ?	 \\
\hlt 36	& 18	& (34,2)	&  \ssz \bf{7}	& \ssz\bf{7}	& \bf{7}	&\ssz --	&\ssz --	& \ssz    --	& \ssz --	&      ?	 \\
\hlt 36	& 36	& (19,17)	&  \ssz 4	& \ssz\bf{7}	& \bf{7}	&\ssz --	&\ssz --	& \ssz    --	& \ssz --	&      ?	 \\
\hlt 36	& 36	& (23,13)	&  \ssz 5	& \ssz\bf{7}	& \bf{7}	&\ssz --	&\ssz --	& \ssz    --	& \ssz -- 	&      ? 	 \\
\hlt 36	& 36	& (25,11)	&  \ssz 5	& \ssz\bf{7}	& \bf{7}	&\ssz --	&\ssz --	& \ssz    --	& \ssz --	&      ?	 \\
\hlt 36	& 36	& (29,7)	&  \ssz 6	& \ssz\bf{7}	& \bf{7}	&\ssz --	&\ssz --	& \ssz    --	& \ssz --	&      ?	 \\
\hlt 36	& 36	& (31,5)	&  \ssz \bf{7}	& \ssz\bf{7}	& \bf{7}	&\ssz --	&\ssz --	& \ssz    --	& \ssz --	&      ?	 \\
\hlt 36	& 36	& (35,1)	&  \ssz 7	& \ssz     7	&\bf{8}$^\ddag$	&\ssz --	&\ssz --	& \ssz    --	& \ssz --	&      ?	 \\\cline{1-11}
\end{tabular}
\end{center}

\ddag : there are no $1$-perfect codes in $H(7,6)$~\cite{GolombPosner64}, so there are no $(35,1)$-colorings in $H(7,6)$.

\end{document}